\documentclass[11pt,twoside,reqno,letter]{amsart}
\usepackage[margin=1in]{geometry}
\usepackage{amsmath,graphicx,varioref,amscd,amssymb,color,bm,amsthm,epsfig,color,enumerate,fancybox,bbm,esint,upref,xcolor,latexsym,mathtools,breqn,appendix,setspace,physics,tikz,pdfsync}
\usepackage[pdftex, colorlinks=true, linktocpage=true]{hyperref}

\makeatletter
\def\@tocline#1#2#3#4#5#6#7{\relax
	\ifnum #1>\c@tocdepth 
	\else
	\par \addpenalty\@secpenalty\addvspace{#2}%
	\begingroup \hyphenpenalty\@M
	\@ifempty{#4}{%
		\@tempdima\csname r@tocindent\number#1\endcsname\relax
	}{%
		\@tempdima#4\relax
	}%
	\parindent\z@ \leftskip#3\relax \advance\leftskip\@tempdima\relax
	\rightskip\@pnumwidth plus4em \parfillskip-\@pnumwidth
	#5\leavevmode\hskip-\@tempdima
	\ifcase #1
	\or\or \hskip 1em \or \hskip 2em \else \hskip 3em \fi%
	#6\nobreak\relax
	\dotfill\hbox to\@pnumwidth{\@tocpagenum{#7}}\par
	\nobreak
	\endgroup
	\fi}

\def\l@subsection{\@tocline{2}{0pt}{2pc}{5pc}{}}
\makeatother

\numberwithin{equation}{section}

\theoremstyle{plain}
\newtheorem{thm}{Theorem}[section]
\newtheorem*{thm*}{Theorem}
\newtheorem{mydef}[thm]{Definition}
\newtheorem{lem}[thm]{Lemma}
\newtheorem*{lem*}{Lemma}

\newtheorem*{prop*}{Proposition}

\theoremstyle{remark}
\newtheorem{rem}[thm]{Remark}

\newcommand{\R}{\mathbb{R}}
\newcommand{\C}{\mathbb{C}}
\newcommand{\N}{\mathbb{N}}

\newcommand{\T}{\mathbb{T}}
\newcommand{\Leray}{\mathcal{P}}

\newcommand{\mi}{m_i}
\newcommand{\Mi}{M_i}
\newcommand{\mf}{m_f}
\newcommand{\Mf}{M_f}
\newcommand{\veps}{\varepsilon}

\begin{document}
	
	
	\title[Small-data global existence in a model of superfluidity]{Small-data global existence of solutions for the Pitaevskii model of superfluidity}

	\author[Jang]{Juhi Jang}
	\address[Jang]{\newline
		Department of Mathematics \\ University of Southern California \\ Los Angeles, CA 90089, USA}
	\email[]{\href{juhijang@usc.edu}{juhijang@usc.edu}}
	
	\author[Jayanti]{Pranava Chaitanya Jayanti}
	\address[Jayanti]{\newline
		Department of Mathematics \\ University of Southern California \\ Los Angeles, CA 90089, USA}
	\email[]{\href{pjayanti@usc.edu}{pjayanti@usc.edu}}
	
	\author[Kukavica]{Igor Kukavica}
	\address[Kukavica]{\newline
		Department of Mathematics \\ University of Southern California \\ Los Angeles, CA 90089, USA}
	\email[]{\href{kukavica@usc.edu}{kukavica@usc.edu}}

	\date{\today}
	
	
	
	
	\maketitle
	
	\begin{abstract}
		We investigate a micro-scale model of superfluidity derived by Pitaevskii in 1959 \cite{Pitaevskii1959PhenomenologicalPoint} to describe the interacting dynamics between the superfluid and normal fluid phases of Helium-4. The model involves the nonlinear Schr\"odinger equation (NLS) and the Navier-Stokes equations (NSE), coupled to each other via a bidirectional nonlinear relaxation mechanism. Depending on the nature of the nonlinearity in the NLS, we prove global/almost global existence of solutions to this system in $\T^2$ -- strong in wavefunction and velocity, and weak in density.
	\end{abstract}

	\setcounter{tocdepth}{2} 
	\tableofcontents
	
	\pagebreak
	
	\section{Introduction}  \label{intro}
	
	Superfluids constitute a phase of matter that is achieved when certain substances are isobarically cooled, resulting in Bose-Einstein condensation. That Helium-4 (and also its isotope Helium-3) undergoes such a quantum mechanical phase transition was first experimentally discovered \cite{Kapitza1938Viscosity-Point,Allen1938FlowII} over 80 years ago and has been the subject of intense inquiry ever since. Despite this, a single theory that describes the phenomenon continues to elude us.
	
	The general picture is that at non-zero temperatures, there is a mixture of two interacting phases: the \textit{normal} fluid and the superfluid \cite{Paoletti2011QuantumTurbulence,Vinen:808382,Vinen2006AnTurbulence,Barenghi2014IntroductionTurbulence,Barenghi2001QuantizedTurbulence,Barenghi2014ExperimentalFluid}. It is important to note that this is not like classical multiphase flow, where one can define a clear boundary between the two phases. Instead, some atoms are in the normal fluid phase, and some are in the superfluid phase, with both fluids occupying the entire volume. The normal fluid is well-modeled by the Navier-Stokes equations (NSE), while the description of the superfluid varies by the length scale that we are interested in (see \cite{Berloff2014ModelingTemperatures,Jayanti2022AnalysisSuperfluidity} for a discussion). Briefly, the superfluid is described by the NSE at large scales \cite{Holm2001}, a vortex model at intermediate scales \cite{Schwarz1978TurbulenceCounterflow,Schwarz1985Three-dimensionalInteractions,Schwarz1988Three-dimensionalTurbulence}, and the nonlinear Schr\"odinger equation (NLS) at small scales \cite{Khalatnikov1969AbsorptionPoint,Carlson1996AVortices}. The macro-scale, NSE-based description is a current topic of numerical research \cite{Verma2019TheModel,Roche2009QuantumCascade,Salort2011MesoscaleTurbulence}, and has also been rigorously analyzed \cite{Jayanti2021GlobalEquations}. In this paper, we use the micro-scale, NLS-based model by Pitaevskii \cite{Pitaevskii1959PhenomenologicalPoint}, which has previously been considered in \cite{Jayanti2022LocalSuperfluidity,Jayanti2022UniquenessSuperfluidity}.
	
	A missing piece of the physics puzzle here is the nature of the interaction mechanism. It is known that the interaction between the fluids is dissipative/retarding. Pitaevskii thus derived a micro-scale model that intertwines the NLS (for the superfluid) and the NSE (for the normal fluid). The coupling is nonlinear, bidirectional and transfers mass, momentum, and energy between the two fluids. For the combined system of both phases, the model respects the conservation of total mass and total momentum, while the total energy decreases in accordance with the dissipation. 
	
	The NLS, in its most popular form, is fundamentally a dispersive partial differential equation with a cubic nonlinearity that models systems with low-energy wave interactions, such as dipolar quantum gases \cite{Carles2008OnGases,Sohinger2011BoundsEquations}. The well-posedness issues of NLS have been tackled in many situations \cite{CollianderWell-posednessEquations}, and its scattering solutions \cite{Tao2006NonlinearAnalysis,Dodson2016Global2} have been of particular interest. The NLS can also be recast as a system of compressible Euler equations (referred to as \textit{quantum hydrodynamics} or QHD) with an additional \textit{quantum pressure} term \cite{Carles2012MadelungKorteweg}. This system is a special case of the more general Korteweg models, subject to much mathematical analysis. Hattori and Li \cite{Hattori1994SolutionsType} showed that the 2D QHD equations are locally well-posed for high-regularity data, and improved this to global well-posedness in the case of small data \cite{Hattori1996GlobalMaterials}. J\"ungel \cite{Jungel2002LocalEquations} established local strong solutions to the QHD-Poisson system, formed by including a potential governed by the Poisson equation. The same model possesses local-in-time classical solutions in 1D when the data is highly regular \cite{Jungel2004QuantumDecay}. For initial conditions close to a stationary state, the solutions are global-in-time and converge exponentially fast to the stationary state. Blow-up criteria have also been derived for QHD \cite{Wang2020AModel,Wang2021AModel}. While the discussion so far has focused on strong solutions, there has also been rising interest in the weak formulation of QHD-like models. Antonelli and Marcati \cite{Antonelli2009OnDynamics,Antonelli2012TheDimensions,Antonelli2015FiniteSuperfluidity} introduced the novel fractional step method in the pursuit of finite-energy global weak solutions. The idea was to revert (from QHD) to NLS, which was easier to solve, and account for collision-induced momentum transfer via periodic updates to the wavefunction. In this process, the occurrence of quantum vortices could also be characterized by imposing irrotationality of the velocity field (away from vacuum regions). Using special test functions that permit better control of the quantum pressure term, J\"ungel \cite{Jungel2010GlobalFluids} proved that the viscous QHD system admits weak solutions in 2D. For small values of viscosity, these solutions were global in time. The proof utilized a redefinition of the velocity that converts the hyperbolic continuity equation into a parabolic one, a technique that was pioneered by Bresch and Desjardins \cite{Bresch2004QuelquesKorteweg} for Korteweg systems in general. Vasseur and Yu \cite{Vasseur2016GlobalDamping} expanded J\"ungel's result to a wider class of test functions while adding some physically-motivated drag terms. Various forms of damping have appeared in the literature, primarily serving two different roles: (i) as an approximating scheme for both the compressible Navier-Stokes with degenerate viscosities \cite{Li2015GlobalViscosities,Vasseur2016ExistenceEquations} as well as Korteweg-type systems \cite{Antonelli2017GlobalEquations,Antonelli2020RelaxationEquation,Antonelli2022GlobalEquations}, and (ii) as a means of proving global existence \cite{Chauleur2020GlobalEquations} or relaxation to a steady state \cite{Bresch2022OnTerm,Su2022ExponentialForce}. Most works involving Korteweg systems use the notion of $\kappa$-entropy that was first demonstrated in \cite{Bresch2015Two-velocityViscosities}. Furthermore, even questions of non-uniqueness (and weak-strong uniqueness) of weak solutions have been addressed for the QHD-Poisson system with linear drag using convex integration \cite{Donatelli2015Well/IllProblems}.
	
	It is only at absolute zero temperature that superfluids can be well-approximated by the use of the NLS alone. For temperatures above zero and below about $2.17$K, we have a mixture of both fluids. In this article, we consider Pitaevskii's model \cite{Pitaevskii1959PhenomenologicalPoint} which couples the NLS and the NSE. The model was initially derived for a fully compressible normal fluid. While compressible fluids are more realistic in some scenarios, they are also much more challenging to both rigorously analyze and numerically simulate. \cite{Feireisl2004DynamicsFluids,Lions1996MathematicalMechanicsb} contain several classical results on the compressible NSE. On the other hand, the incompressible NSE (no density equation) is arguably the most studied nonlinear partial differential equation in mathematics (see \cite{Temam1977Navier-StokesAnalysis,Majda2002VorticityFlow,Robinson2016TheEquations} for classical results). In this article, we approximate the normal fluid to be incompressible, but the density persists, varying from point to point in the flow domain. What results is an incompressible, inhomogeneous flow: compressible NSE appended with the condition of divergence-free velocity. This model of fluids was first investigated by Kazhikov for local weak solutions when the initial density is bounded from below \cite{Kazhikov1974Fluid}, and vacuum states were allowed in an improvement by Kim \cite{Kim1987WeakDensity}. Further advances for weak solutions were made by Simon \cite{Simon1990NonhomogeneousPressure}, who in particular analyzed their continuity at $t=0$, and also proved the existence of global solutions in a less regular space. Meanwhile, Ladyzhenskaya and Solonnikov \cite{Ladyzhenskaya1978UniqueFluids} presented the case for strong solutions: With the density bounded from below, it is possible to construct local (global) unique solutions in 3D ~(2D). Furthermore, if the data is small enough, one obtains global-in-time unique solutions. Results in the same spirit were proven by Danchin for small perturbations from the stationary state in critical Besov spaces \cite{Danchin2003Density-dependentSpaces}. He further established the inviscid limit of the incompressible inhomogeneous NSE  in subcritical spaces \cite{Danchin2006TheFluids}. The local existence theorem by Ladyzhenskaya and Solonnikov was shown to be valid for non-negative densities as long as the initial data satisfied a compatibility criterion \cite{Choe2003StrongFluids}. This work by Choe and Kim has since spurred on several other results that utilize such compatibility conditions on the initial data.
	
	Given the immense interest in the NLS and NSE, the rigorous study of a coupled system should be a natural next step. Indeed, one such two-fluid model of superfluidity was analyzed by Antonelli and Marcati in \cite{Antonelli2015FiniteSuperfluidity}. The superfluid was described by the NLS, and the normal fluid by the compressible NSE. This is similar to the system considered in this article, save for two key differences. Firstly, their model did not permit any mass transfer between the two fluids (which allows for global-in-time solutions). As we shall discuss, this is the biggest roadblock in Pitaevskii's model and essentially defines the strategy used. Secondly, the momentum transfer in their model is unidirectional and linear, affecting only the superfluid phase (as opposed to the bidirectional and nonlinear nature of the coupling in this work).
	
	Thanks to the retarding interactions between the two phases, the NLS acquires a dissipative flavor and renders it parabolic. This lets us extract dissipative contributions to the energy estimates. To analyze the momentum equation of the NSE, we work with initial velocity in $H^1_{\text{d}}$. This yields appropriate regularity for the velocity, in order to adequately control the relaxation mechanism which contains quadratic terms in the velocity. Parting ways from \cite{Kim1987WeakDensity}, we begin with an initial density field that is bounded from below. This is necessary since the continuity equation is unusual and is not a homogeneous transport equation. Our primary goal is to avoid the occurrence of zero or negative densities at any time. To this end, we must limit the effect of inhomogeneity, which is the relaxation mechanism that allows for mass and momentum transfer between the two fluids. As a serendipitous by-product of this non-zero density field, we also obtain control of $\norm{\partial_t u}_{L^2_t L^2_x}$, which allows the use of compactness arguments to actually obtain strong continuity in time of the velocity field. 
	
	The crux of this work is to derive a priori estimates and carefully extract coercive terms that allow for norms to decay, while avoiding any derivatives on the density of the normal fluid. To engineer this decay, we include a linear drag term for the NSE. Additionally, we also present results for any polynomial-type nonlinearity in the NLS. We now mention the notation used in the article before describing the model and stating the results. 
	
	\subsection{Notation} \label{notation}
	We denote by $H^s(\T^2)$ the completion of $C^{\infty}(\T^2)$ under the Sobolev norm $H^s$, while we use $\Dot{H}^s(\T^2)$ when referring to the homogeneous Sobolev spaces. Consider a 2D vector-valued function $u\equiv (u_1,u_2)$, where $u_i\in C^{\infty}(\T^2)$ for $i=1,2$. The set of all divergence-free, smooth 2D functions $u$ defines $C^{\infty}_{\text{d}}(\T^2)$. Then, $H^s_{\text{d}}(\T^2)$ is the completion of $C^{\infty}_{\text{d}}(\T^2)$ under the $H^s$ norm.
	
	The $L^2$ inner product, denoted by $\langle \cdot,\cdot \rangle$, is sesquilinear (the first argument is complex conjugated, indicated by an overbar) to accommodate the complex nature of the Schr\"odinger equation, i.e., $\langle \psi,\varphi \rangle := \int_{\T^2} \Bar{\psi}\varphi \ dx$. Since the velocity and density are real-valued functions, we ignore the complex conjugation when they constitute the first argument of the inner product.
	
	We use the subscript $x$ to denote Banach spaces that are defined over $\T^2$. For instance, $L^p_x := L^p(\T^2)$ and $H^s_{\text{d},x} := H^s_{\text{d}}(\T^2)$. For spaces/norms over time, the subscript $t$ denotes the time interval in consideration, such as $L^p_t := L^p_{[0,t]}$. The Bochner spaces $L^p(0,T;X)$ and $C([0,T];X)$ have their usual meanings, as $L^p$ and continuous maps (respectively) from $[0,T]$ to a Banach space $X$. 
	
	We also use the notation $X\lesssim Y$ and $X\gtrsim Y$ to imply that there exists a positive constant $C$ such that $X\le CY$ and $CX\ge Y$, respectively. When appropriate, the dependence of the constant on various parameters shall be denoted using a subscript as $X\lesssim_{k_1,k_2} Y$ or $X\le C_{k_1,k_2}Y$. Throughout the article, $C$ is used to denote a (possibly large) constant that depends on the system parameters listed in ~\eqref{small data condition statement}, while $\kappa$ and $\veps$ are used to represent (small) positive numbers. The values of $C$, $\kappa$, and $\veps$ can vary across the different steps of calculations.

	\subsection{Organization of the paper}
	
	In Section ~\ref{mathematical model}, we present and discuss the mathematical model, along with statements of the main results. Several a priori estimates, at increasing levels of regularity, are derived in Section ~\ref{a priori estimates}. The construction of the semi-Galerkin scheme and the renormalization of the density are discussed in Section ~\ref{existence proof}.
	
	\section{Mathematical model and main results} \label{mathematical model}
	The superfluid phase is described by a complex wavefunction, whose dynamics are governed by the nonlinear Schr\"odinger equation (NLS), while the normal fluid is modeled using the compressible Navier-Stokes equations (NSE). In all generality, the full set of equations can be found in \cite[Section 2]{Pitaevskii1959PhenomenologicalPoint}. In what follows, we use a slightly simplified and modified version of the equations, arrived at by making the following assumptions.
	
	\begin{enumerate}
		\item We consider a general power-law nonlinearity for the NLS. This is done by choosing the internal energy density of the system to be $\frac{2\mu}{p+2}\abs{\psi}^{p+2}$, for $1\le p< \infty$ (see Remark \ref{restricting p>=1}). We also assume that the internal energy is independent of the density of the normal fluid.
		
		\item We work in the limit of a divergence-free normal fluid velocity. This means that the pressure is a Lagrange multiplier, rendering the equations of state and entropy unnecessary. Note that, due to the nature of the coupling between the two phases, the density of the normal fluid is not simply transported.
		
		\item A linear drag term has been included in the momentum equation to account for the lack of coercive estimates for the velocity.
		
		\item Planck's constant $(\hbar)$ and the mass of the Helium atom $(m)$ have both been set to unity for simplicity.
	\end{enumerate}
	
	We now state the equations used in this paper:
	\begin{align}
		\partial_t \psi + \lambda B\psi &= -\frac{1}{2i}\Delta\psi + \frac{\mu}{i}\lvert\psi\rvert^p \psi \tag{NLS} \label{NLS} \\
		B = \frac{1}{2}\left(-i\nabla - u \right)^2 + \mu \lvert \psi \rvert^p &= -\frac{1}{2}\Delta + \frac{1}{2}\lvert u \rvert^2 + iu\cdot\nabla + \mu \lvert \psi \rvert^p \tag{CPL} \label{coupling} \\
		\partial_t \rho + \nabla\cdot(\rho u) &= 2\lambda\Re(\Bar{\psi}B\psi) \tag{CON} \label{continuity} \\
		\partial_t (\rho u) + \nabla\cdot (\rho u \otimes u) + \nabla q - \nu \Delta u + \alpha\rho u &= \!\begin{multlined}[t]
			-2\lambda \Im(\nabla \Bar{\psi}B\psi) + \lambda\nabla \Im(\Bar{\psi}B\psi) \tag{NSE} \label{NSE} + \frac{\mu}{2}\nabla\lvert\psi\rvert^{p+2} 
		\end{multlined} \\
		\nabla\cdot u &= 0. \tag{DIV} \label{divergence-free}
	\end{align}
	Here, $\psi$ is the wavefunction describing the superfluid phase, while $\rho$, $u$, and $q$ are the density, velocity and pressure (respectively) of the normal fluid. The normal fluid has viscosity $\nu$ and drag coefficient $\alpha$, while $\mu$ (positive constant) is the strength of the scattering interactions within the superfluid\footnote{$\mu>0$ (resp. $\mu<0$) is called the defocusing (resp. focusing) NLS.}. This scattering nonlinearity has an exponent $p\in [1,\infty)$. Finally, $\lambda$ is a positive constant that indicates the coupling strength between the two phases. The coupling is denoted by the nonlinear operator $B$.
	
	The Schr\"odinger equation dictates the evolution of the wavefunction, generated via the action of the Hamiltonian (roughly, the energy) of the system. The coupling $B$ resembles the relative kinetic energy\footnote{There is also the nonlinear wavefunction term, so that the relaxation to equilibrium also depends on the potential energy of the superfluid.} between the two phases. This is evident upon recalling that the quantum mechanical momentum operator (in the position basis) is $-i\hbar\nabla$. The purpose of this coupling is to allow for mass/momentum transfer between the two phases as a means of relaxation or dissipation.
	
	These equations are supplemented with the initial conditions
	\begin{equation*} \tag{INI} \label{initial conditions}
		\psi(0,x) = \psi_0(x), \quad u(0,x) = u_0(x), \quad \rho(0,x) = \rho_0(x) \quad \text{a.e. } x\in\T^2 .
	\end{equation*}
	We use periodic boundary conditions, i.e., we are working on the two-dimensional torus $[0,1]^2$.

    \subsection{Weak solutions and the existence theorems} \label{weak solutions and the existence theorems}
	Having stated the model, the notion of weak solutions to ~\eqref{NLS}, ~\eqref{NSE}, ~\eqref{continuity}, and ~\eqref{divergence-free} (with initial conditions ~\eqref{initial conditions} and periodic boundary conditions), henceforth referred to as the \emph{Pitaevskii model}, is as follows.
	
	\begin{mydef}[Weak solutions\footnote{See Remark ~\ref{strong or weak solutions?}.}] \label{definition of weak solutions}
		For a given time $T>0$, a triplet $(\psi,u,\rho)$ is called a weak solution to the Pitaevskii model if the following conditions hold.
		\begin{enumerate} [(i)]
			\item 
			$\psi\in L^{\infty}([0,T];H^2(\T^2))\cap L^2(0,T;H^3(\T^2))$, $u\in L^{\infty}([0,T];H^1_{\text{d}}(\T^2))\cap L^2(0,T;H^2_{\text{d}}(\T^2))$, and $\rho~\in~L^{\infty}([0,T]\times\T^2)$, and
			
			\item $\psi$, $u$, and $\rho$ satisfy the governing equations in the sense of distributions, i.e., for all test functions $\varphi$, $\Phi$, and $\sigma$ described below, we have
			\begin{equation} \label{weak solution wavefunction}
				\begin{aligned}
					&-\int_0^T \int_{\T^2} \left( \psi\partial_t\Bar{\varphi} + \frac{1}{2i}\nabla\psi\cdot\nabla\Bar{\varphi} - \lambda\Bar{\varphi}B\psi - i\mu\Bar{\varphi}\lvert \psi \rvert^p\psi \right) dx \ dt \\ 
					&\quad = \int_{\T^2} \Big( \psi_0\Bar{\varphi}(0) - \psi(T)\Bar{\varphi}(T) \Big) dx,
				\end{aligned}
			\end{equation}
			with
			\begin{equation} \label{weak solution velocity}
				\begin{aligned}
					&-\int_0^T \int_{\T^2} \Big( \rho u\cdot \partial_t \Phi + \rho u\otimes u:\nabla\Phi - \nu\nabla u:\nabla\Phi - 2\lambda\Phi\cdot\Im(\nabla\Bar{\psi}B\psi) + \alpha\rho u\cdot\Phi \Big) dx \ dt \\ 
					&\quad = \int_{\T^2} \Big( \rho_0 u_0 \Phi(0) - \rho(T)u(T)\Phi(T) \Big) dx
				\end{aligned}
			\end{equation}
			and
			\begin{equation} \label{weak solution density}
				-\int_0^T \int_{\T^2} \Big( \rho \partial_t \sigma + \rho u\cdot\nabla\sigma + 2\lambda \sigma \Re(\Bar{\psi}B\psi) \Big) dx \ dt = \int_{\T^2} \Big( \rho_0 \sigma(0) - \rho(T)\sigma(T) \Big) dx,
			\end{equation}
			where $\psi_0 \in H^2(\T^2)$, $u_0 \in H^1_{\text{d}}(\T^2)$ and $\rho_0 \in L^{\infty}(\T^2)$ are the initial data. The test functions are:
			\begin{enumerate}
				\item a complex-valued scalar field $\varphi \in H^1(0,T;L^2(\T^2))\cap L^2(0,T;H^1(\T^2))$,
				\item a real-valued, divergence-free (2D) vector field $\Phi \in H^1(0,T;L^2_{\text{d}}(\T^2)) \cap L^2(0,T;H^1_{\text{d}}(\T^2))$, and
				\item a real-valued scalar field $\sigma \in H^1(0,T;L^2(\T^2))\cap L^2(0,T;H^1(\T^2))$.
			\end{enumerate}
		\end{enumerate}
	\end{mydef}
	
	\begin{rem} \label{gradient terms in NSE}
		We note that the last two terms in ~\eqref{NSE} are gradients, just like the pressure term, and thus vanish in the definition of the weak solution (since the test function is divergence-free). Henceforth, we absorb these two gradient terms into the pressure, relabeling the new pressure as $q$.
	\end{rem}
	
	We are now ready to state our main results.
	\begin{thm} [Global existence] \label{global existence}
		Fix any $p\in [1,4)$, and let $\psi_0 \in H^{\frac{5}{2}}(\T^2)$ with $u_0\in H^{1}_{\text{d}}(\T^2)$. Suppose $0< \mi\le\rho_0\le \Mi <\infty$ a.e.\ in $\T^2$. Then, there exist a global weak solution $(\psi,u,\rho)$ to the Pitaevskii model such that the density is bounded between $\mf\in (0,\mi)$ and $\Mf := \Mi + \mi - \mf$, if the initial data satisfy the smallness criterion
		\begin{equation} \label{small data condition statement}
			\norm{\psi_0}_{H^{\frac{5}{2}}_x} + \norm{u_0}_{H^1_x} + \norm{\psi_0}_{L^{p+2}_x} \le \veps_0(\lambda,\mu,\nu,\mi,\Mi,\mf,\alpha,p).
		\end{equation}
		Also, the solution has the regularity
		\begin{gather}
			\psi\in C([0,\infty);H^{\frac{5}{2}}(\T^2)) \cap L^2(0,\infty;H^{\frac{7}{2}}(\T^2)), \label{weak solution psi regularity} \\
			u\in C([0,\infty);H^1_{\text{d}}(\T^2)) \cap L^2(0,\infty;H^{2}_{\text{d}}(\T^2)), \label{weak solution u regularity} \\
			\rho\in L^{\infty}([0,\infty)\times \T^2)\cap C([0,\infty);L^r(\T^2)) , \label{weak solution rho regularity}
		\end{gather}
		for $1\le r<\infty$. Additionally, the solution also satisfies the energy equality
		\begin{equation} \label{energy equality for weak solutions}
			\begin{aligned}
				& \frac{1}{2}\norm{\sqrt{\rho(t)}u(t)}_{L^2_x}^2 + \frac{1}{2}\norm{\nabla\psi(t)}_{L^2_x}^2 + \frac{2\mu}{p+2}\norm{\psi(t)}_{L^{p+2}_x}^{p+2} \\
				&\quad + \nu\norm{\nabla u}_{L^2_{[0,t]}L^2_x}^2 + \alpha\norm{\sqrt{\rho} u}_{L^2_{[0,t]}L^2_x}^2 + 2\lambda\norm{B\psi}_{L^2_{[0,t]}L^2_x}^2 \\ 
				&= \frac{1}{2}\norm{\sqrt{\rho_0}u_0}_{L^2_x}^2 + \frac{1}{2}\norm{\nabla\psi_0}_{L^2_x}^2 + \frac{2\mu}{p+2}\norm{\psi_0}_{L^{p+2}_x}^{p+2} \quad a.e. \ t\in [0,\infty).
			\end{aligned}
		\end{equation}
	\end{thm}
	
	For the case of higher-order nonlinearities, i.e., when $p\ge 4$, we obtain ``almost global" existence.
	
	\begin{thm} [Almost global existence] \label{almost global existence}
		In the case of $p=4$, the solution to the Pitaevskii model has the same regularity properties as in Theorem ~\ref{global existence}, except that their existence is guaranteed on $[0,T]$ such that $T \sim \exp (\veps^{-\frac{1}{2}})$, where $\veps$ is the size of the (sufficiently small) initial data.
		
		\noindent For $p>4$, the existence time scales polynomially with the size of the data, as $T\sim \veps^{-\frac{p}{p-4}}$. In both cases, these solutions also satisfy the energy equality on $[0,T]$.
	\end{thm}
	
	While deriving the a priori estimates, we have to distinguish between the cases $1\le p<2$, $p=2$, $2<p<4$, $p=4$, and $p>4$. This is due to the poor control we have on the superfluid mass. Given that we are on $\T^2$, and our equations do not preserve functions with vanishing mean, the $L^2$ norm becomes the limiting factor even in the decay of higher norms. In the case of the wavefunction, this corresponds to the mass of the superfluid. Similarly, for the velocity, we do not get coercive estimates from the viscosity term alone, at least at the level of the kinetic energy estimate. Thus, we introduce a linear drag term.

	\begin{rem} \label{restricting p>=1}
        Since the self-interaction term in ~\eqref{NLS} involves a discontinuity due to the complex magnitude, evaluating the $H^2$ norm as in~\eqref{psi^p psi in highest norm} requires $p\ge 1$. In particular, points of superfluid vacuum $(\psi=0)$ may lead to problems. As an illustration, consider $D^2 \left(\abs{f}^p f\right)$ for a real-valued function $f$, which can be regularized as $D^2 \left((f^2 + \veps)^{\frac{p}{2}} f\right)$. Upon differentiation, the most problematic term is $(f^2 + \veps)^{\frac{p}{2}-2} f^3 (Df)^2$. To be able to handle this term in the limit $\veps\rightarrow 0$, at the points where $f=0$, we require that $2\left(\frac{p}{2}-2 \right)+3=p-1\ge 0$. This argument can be easily extended to a complex-valued function.
    \end{rem}
	
	\begin{rem} \label{strong or weak solutions?}
		The regularity of the solutions seem to suggest that the wavefunction and velocity are strong solutions. Indeed this is true, as they are strongly continuous in their topologies. On the other hand, the density is truly a weak solution and is the reason for referring to the triplet as a weak solution. This low regularity of the density influences the nature of the calculations that are employed. 
	\end{rem} 
	
	The proofs of both Theorems ~\ref{global existence} and ~\ref{almost global existence} follow from detailed a priori estimates, and a semi-Galerkin scheme to construct the solutions. The a priori estimates only differ slightly for various ranges of the values of $p$, as will be illustrated. The general approach to the problem is motivated by that of \cite{Kim1987WeakDensity}, but we do not allow the density to vanish anywhere. This is because the presence of $u$ in the nonlinear coupling means we are required to control it in $L^{\infty}(\T^2)$ to prevent the formation of vacuum (and regions of negative density). Beginning from the usual mass and energy estimates, we derive a hierarchy of several energies for the wavefunction and velocity.

    \subsection{Significance of the results} \label{significance of the results}
    The holy grail of superfluid modeling is to find a unified description that works at all length scales, and rigorous validation of any proposed models is crucial to this process. The thrust of this paper is the analysis of Pitaevskii's description of superfluidity, the most important feature of which is to characterize the mass transfer between the two fluids. In the course of proving the main theorems, we quantify the conversion of superfluid into normal fluid (Lemma~\ref{lem:superfluid mass estimate}), confirming the interaction-induced relaxation mechanism. We establish the validity of the model in the limit $t\rightarrow\infty$ even as the superfluid mass decreases (polynomially) quickly. The transition in the behavior of the solutions, from global to almost-global, as the self-interactions are increased in strength, is in accordance with the decreasing mass decay. However, the threshold $p=4$ still begs for a physical explanation. Of the assumptions underlying our theorems, relaxing the demands of small data and positive normal fluid density would be important future advancements in the context of the Pitaevskii model.

    The rigorous analysis of superfluid models is a fairly new topic, and we expect for this work to pave the way for further results in this direction. Some questions of interest, particularly of consequence to physicists and engineers, are the issues of stability and compressibility. For example, in~\cite{Pitaevskii1959PhenomenologicalPoint}, Pitaevskii investigated the propagation of sound waves in superfluid Helium by studying the case when the superfluid has only small density gradients. It has to be noted that his derivation of the model accounted for the contributions to the internal energy of the system from both fluids. Thus, by utilizing appropriate self-interactions (for instance, non-local potentials, or including the normal fluid density), it would be important to test the model against experimental findings. A mathematical guarantee of the existence of solutions to the Pitaevskii model is essential to complement the efforts to numerically simulate such complicated systems~\cite{Brachet2023CouplingFlows}. It is worth mentioning that a better understanding of superfluidity could be revolutionary to most modern experiments in physics (including the Large Hadron Collider~\cite{Lebrun1994SuperfluidCERN,Rousset2018EvaluationD2}), and also to the fields of quantum computing~\cite{Hollister2021AApplications}, gravitational wave astronomy~\cite{Singh2017Detecting4He}, and dark matter~\cite{vonKrosigk2023DELight:Helium}. All of these use helium as a cryogen, often as a superfluid-normal fluid mixture due to the superfluid's excellent thermal conductivity~\cite{Vinen:808382}.

	\subsection{The strategy} \label{the strategy}
	The nonlinear coupling terms in ~\eqref{NLS} and ~\eqref{NSE} may be the most obvious differences between this model and other standard fluid dynamics models, but the source term in ~\eqref{continuity} is the most troublesome. The backbone of our approach towards proving global existence is ensuring a positive lower bound for the density at all times. This involves a meticulous handling of the a priori estimates so as to obtain coercive terms that lead to global-in-time bounds. Throughout the calculations, we ensure that the density norms are only in Lebesgue spaces: $\rho$ is not smooth enough to be differentiated (even weakly). Before we outline the strategy, we discuss some properties of the coupling operator $B$. Henceforth, we refer to the linear (in $\psi$) part of $B$ as $B_L$. Thus, 
	\begin{equation} \label{defining B_L}
		B_L = B - \mu\lvert\psi\rvert^p = -\frac{1}{2}\Delta + \frac{1}{2}\lvert u\rvert^2 + iu\cdot\nabla .
	\end{equation}
	
	\begin{lem} [$B_L$ is symmetric and $B$ is coercive] \label{coupling symmetric and non-negative} We have
		\begin{enumerate}
			\item $\langle \phi,B_L\psi \rangle = \langle B_L\phi,\psi \rangle \text{ for all } \phi,\psi\in H^1(\T^2)$,
			\item $\Re\langle \psi,B\psi \rangle \ge \mu\lVert\psi\rVert^{p+2}_{L^{p+2}_x} \text{ for all } \psi\in H^1(\T^2)$.
		\end{enumerate}
	\end{lem}
	\proof  Both calculations follow using integration by parts.
	\begin{enumerate}
		\item By ~\eqref{defining B_L} and incompressibility of $u$, we have
		\begin{align*}
			\langle \phi,B_L\psi \rangle &= \int_{\T^2} \Bar{\phi} B_L\psi = \int_{\T^2} \Bar{\phi} \left[ -\frac{1}{2}\Delta\psi + \frac{1}{2}\lvert u\rvert^2\psi + iu\cdot\nabla\psi \right] \\
			&= \int_{\T^2} \left[ -\frac{1}{2}\Delta\Bar{\phi} + \frac{1}{2}\lvert u\rvert^2\Bar{\phi} - iu\cdot\nabla\Bar{\phi} \right]\psi = \int_{\T^2} (\overline{B_L\phi})\psi = \langle B_L\phi,\psi \rangle .
		\end{align*}
		\item Similarly,
		\begin{align*}
			\Re \langle \psi,B\psi \rangle &= \Re\int_{\T^2} \Bar{\psi} B\psi = \Re\int_{\T^2} \Bar{\psi} \left[ -\frac{1}{2}\Delta\psi + \frac{1}{2}\lvert u\rvert^2\psi + iu\cdot\nabla\psi + \mu\lvert\psi\rvert^p\psi \right] \\
			&= \frac{1}{2}\lVert\nabla\psi\rVert^2_{L^2_x} + \frac{1}{2}\int_{\T^2}\lvert u\rvert^2 \lvert\psi\rvert^2 - \Im\int_{\T^2} u\Bar{\psi}\cdot\nabla\psi +  \mu\lVert\psi\rVert^{p+2}_{L^{p+2}_x} \\
			&\ge \mu\lVert\psi\rVert^{p+2}_{L^{p+2}_x} .
		\end{align*}
		In the last inequality, we used H\"older's and Young's inequalities to cancel the third term with the first two terms:
		\begin{gather*}
			-\Im\int_{\T^2} u\Bar{\psi}\cdot\nabla\psi \ge -\frac{1}{2}\lVert u\psi \rVert^2_{L^2_x} - \frac{1}{2}\lVert\nabla\psi\rVert^2_{L^2_x} .
		\end{gather*}
		\qed
	\end{enumerate}
	
	\begin{rem}
		Given that $B$ provides a relaxation mechanism, it is tempting to treat it, or at least its linear part $B_L$, as a dissipative second-order elliptic operator whose eigenfunctions can be used as a basis for the semi-Galerkin scheme. Even though $B_L$ is symmetric and has a non-negative real part, this cannot work since it has time-dependent coefficients, and so its eigenvalues and eigenfunctions also depend on time. Moreover, $B_L$ does not have a spectral gap at $0$. Its eigenvalues are not known to be bounded from below by a positive number.
	\end{rem}
	
	Thus, by integrating ~\eqref{continuity} over $\T^2$, the advective term vanishes and using Lemma ~\ref{coupling symmetric and non-negative}, we have
	\begin{equation}
		\frac{d}{dt} \int_{\T^2}\rho \ dx = 2\lambda \Re\int_{\T^2}\Bar{\psi}B\psi \ge 0 .
	\end{equation} 
	This implies that the overall mass of the normal fluid does not decrease with time. Put differently, the coupling causes superfluid to be converted into normal fluid, \textit{on average}. However, the RHS of ~\eqref{continuity} need not be non-negative pointwise in $\T^2$. So it is not inconceivable that the density of the normal fluid may locally vanish, or even take negative values! To prevent physically unrealistic density fields, and because our estimates require a strictly positive density, we fix a positive lower bound for $\rho$. Based on this, we define our existence time $T$, so that $\rho$ does not drop below the lower bound until time $T$. Our goal is to show that this lower bound can be maintained for arbitrarily long, provided we begin from sufficiently small data.  
	
	\begin{mydef} [Existence time] \label{existence time definition}
		Start with an initial density field $0<\mi \le\rho_0(x)\le \Mi <\infty$. Given $0<\mf<\mi$, we define the existence time for the solution as
		\begin{equation} \label{abstract definition existence time}
			T_* := \inf \{t>0 \ \lvert \ \inf_{\T^2}\rho(t,x) = \mf\} .
		\end{equation}
	\end{mydef} 
	A \textit{formal} solution to the continuity equation can be written using the method of characteristics. Let $X_{\alpha}(t)$ be the characteristic starting at $\alpha\in\T^2$. To wit, the characteristic solves the differential equation
	\begin{equation} \label{characteristics}
		\begin{aligned}
			\frac{d}{dt}X_{y}(t) &= u(t,X_y(t)) \\
			X_y(0) &= y \in\T^2 .
		\end{aligned}
	\end{equation} 
	Here, $u$ is the velocity of the normal fluid. So, along such characteristics,
	\begin{equation} \label{density along characteristic}
		\rho(t,X_y(t)) = \rho_0(y) + 2\lambda\Re\int_0^t \Bar{\psi}B\psi (\tau,X_y(\tau)) \ d\tau .
	\end{equation} 
	From ~\eqref{abstract definition existence time} and ~\eqref{density along characteristic}, it is clear that a sufficient condition to ensure the density is bounded from below by $\mf$ is
	\begin{equation} \label{constraint for positive density}
		2\lambda\int_0^{T} \lvert\Bar{\psi}B\psi\rvert (\tau,X_y(\tau)) \ d\tau \le \mi-\mf ,
	\end{equation}
	for all $T\le T_*$. This can in turn be ensured through the sufficiency
	\begin{equation} \label{constraint to choose existence time}
		2\lambda \lVert\psi\rVert_{L^2_{[0,T]}L^{\infty}_x} \lVert B\psi\rVert_{L^2_{[0,T]}L^{\infty}_x} \le \mi-\mf .
	\end{equation}
	So, we are looking to show that ~\eqref{constraint to choose existence time} $-$ actually a stronger version of it $-$ holds irrespective of $T$, so that we can conclude that the density is always greater than $\mf$. This is achieved by selecting small enough data, and allows us to deduce the global existence of solutions. Since $B\psi$ involves a second-order derivative, its $L^{\infty}_x$ boundedness leads us to high-regularity spaces. The momentum equation ~\eqref{NSE} is used to estimate $\norm{u}_{L^2_t H^2_x}$ and $\norm{u}_{L^2_t H^1_x}$, which are useful in handling parts of $\norm{B\psi}_{L^{\infty}_x}$. As a by-product of these calculations, we are also able to bound $\norm{\partial_t u}_{L^2_t L^2_x}$, which plays a part in the compactness arguments for the strong time-continuity of $u$. The Schr\"odinger equation ~\eqref{NLS} is used to derive increasingly higher-order a priori estimates of $\psi$. In all these calculations, we work with density that is only in $L^{\infty}_x$.

	\section{A priori estimates} \label{a priori estimates}
	
	Throughout this section, we derive the required a priori estimates, using formal calculations. We assume the wavefunction and velocity are smooth functions and that the density is bounded from below by $\mf>0$ in $[0,T]$. Here, $T$ is any time less than the local existence time $T_*$, and is extended to global existence in Section ~\ref{ensuring positive density}. 
	
	\subsection{Superfluid mass estimate}
	
	\begin{lem}[Algebraic decay rate of superfluid mass] \label{lem:superfluid mass estimate}
		The mass of the superfluid 
		\begin{equation*}
			S(t) := \norm{\psi(t)}_{L^2_x}^2
		\end{equation*} decays algebraically in time as $(1+t)^{-\frac{2}{p}}$, and is bounded from above by the initial mass ~$S_0$.
	\end{lem} 
	
	\begin{proof}
		Multiplying ~\eqref{NLS} by $\Bar{\psi}$, taking the real part, and integrating over $\T^2$ gives
		\begin{equation} \label{superfluid mass estimate initial}
			\frac{1}{2}\frac{d}{dt}\norm{\psi}_{L^2_x}^2 + \lambda \int_{\T^2}\Re (\Bar{\psi}B\psi) = 0.
		\end{equation} 
		The Laplacian term on the RHS of ~\eqref{NLS} vanishes using integration by parts. By Lemma ~\ref{coupling symmetric and non-negative}, the second term in ~\eqref{superfluid mass estimate initial} is bounded from below by the $L^{p+2}_x$ norm, so we get
		\begin{equation} \label{superfluid mass equation coercive}
			\frac{1}{2}\frac{d}{dt} \norm{\psi}_{L^2_x}^2 + \lambda \mu \norm{\psi}_{L^{p+2}_x}^{p+2} \le 0.
		\end{equation}
		Since we are in a domain of unit volume, H\"older's inequality leads to
		\begin{equation} \label{superfluid mass estimate}
			\frac{d}{dt} \frac{1}{2}\norm{\psi}_{L^2_x}^2 + \lambda \mu \norm{\psi}_{L^2_x}^{p+2} \le 0.
		\end{equation}
		It is now easy to conclude that the mass of superfluid (using the quantum mechanical interpretation of the wavefunction) decays algebraically in time. Namely,
		\begin{equation} \label{superfluid mass bound}
			S(t) = \norm{\psi(t)}_{L^2_x}^2 \lesssim \frac{S_0}{\left(1+S_0^{\frac{p}{2}} t\right)^{\frac{2}{p}}} \quad , \quad t \in [0,T],
		\end{equation} 
		where $S_0 := \norm{\psi_0}_{L^2_x}^2$ is the initial mass of the superfluid.
	\end{proof}

	\subsection{Energy estimate} \label{energy estimate}
	In this subsection (Section ~\ref{energy estimate}), we derive the governing equations for the energy 
	\begin{equation} \label{energy definition}
		E(t) := \frac{1}{2}\norm{\sqrt{\rho(t)}u(t)}_{L^2_x}^2 + \frac{1}{2}\norm{\nabla\psi(t)}_{L^2_x}^2 + \frac{2\mu}{p+2}\norm{\psi(t)}_{L^{p+2}_x}^{p+2}.
	\end{equation}
	In Section ~\ref{higher order estimate}, we work with a higher-order energy $X(t)$, combining it with $E(t)$ in Section ~\ref{gronwall inequality step for higher order estimate}. We begin by acting with the gradient operator on ~\eqref{NLS}, multiplying by $\nabla\Bar{\psi}$, and taking the real part. This gives
	\begin{equation*} \label{energy estimate first step}
		\frac{1}{2}\partial_t \abs{\nabla\psi}^2 = -\frac{1}{2}\Im(\nabla\Bar{\psi}\cdot\nabla\Delta\psi) - \lambda\Re(\nabla\Bar{\psi}\cdot\nabla(B\psi)) - \mu\nabla\abs{\psi}^p\cdot\Im(\Bar{\psi}\nabla\psi).
	\end{equation*} 
	Integrating over $\T^2$, we observe that the first term on the RHS vanishes upon integration by parts due to the periodic boundary conditions. The second term on the RHS is similarly integrated by parts to yield
	\begin{equation} \label{nabla psi norm equation}
		\frac{1}{2}\frac{d}{dt}\norm{\nabla\psi}_{L^2_x}^2 = \lambda\Re\int_{\T^2}\Delta\Bar{\psi}B\psi - \mu\Im\int_{\T^2}\nabla\abs{\psi}^p\cdot\Bar{\psi}\nabla\psi.
	\end{equation} 
	Now, we rewrite the first term on the RHS by expressing the Laplacian in terms of the operator $B$, giving us a dissipative contribution to the energy estimate. Namely,
	\begin{equation} \label{exchanging derivatives for B}
		\begin{aligned}
			\lambda\Re\int_{\T^2}\Delta\Bar{\psi}B\psi &= -2\lambda\Re\int_{\T^2} \left( \overline{B\psi} - \frac{1}{2}\abs{u}^2\Bar{\psi} + iu\cdot\nabla\Bar{\psi} - \mu\abs{\psi}^p\Bar{\psi} \right)B\psi \\
			&= -2\lambda\norm{B\psi}_{L^2_x}^2 + \lambda\int_{\T^2}\abs{u}^2\Re(\Bar{\psi}B\psi) + 2\lambda\int_{\T^2}u\cdot\Im(\nabla\Bar{\psi}B\psi) \\ 
			&\qquad + 2\mu\lambda\int_{\T^2}\abs{\psi}^p\Re(\Bar{\psi}B\psi).
		\end{aligned}
	\end{equation}
	We also have to account for the potential (self-interaction) energy of the wavefunction. To obtain this, we multiply ~\eqref{NLS} by $2\Bar{\psi}$ and take the real part to obtain
	\begin{equation*}
		\partial_t \abs{\psi}^2 + \nabla\cdot\Im(\Bar{\psi}\nabla\psi) = -2\lambda\Re(\Bar{\psi}B\psi).
	\end{equation*} 
	Multiplying the above equation with $\mu\abs{\psi}^p$ and integrating over $\T^2$ leads to
	\begin{equation} \label{cubic nonlinearity potential energy}
		\frac{2\mu}{p+2}\frac{d}{dt}\norm{\psi}_{L^{p+2}_x}^{p+2} - \mu\int_{\T^2}\Im(\Bar{\psi}\nabla\psi)\cdot\nabla\abs{\psi}^p = -2\mu\lambda\int_{\T^2}\abs{\psi}^p\Re(\Bar{\psi}B\psi) .
	\end{equation} 
	Combining ~\eqref{nabla psi norm equation}, ~\eqref{exchanging derivatives for B}, and ~\eqref{cubic nonlinearity potential energy} gives the energy equation for the superfluid,
	\begin{equation} \label{superfluid energy equation}
		\frac{d}{dt}\left( \frac{1}{2}\norm{\nabla\psi}_{L^2_x}^2 + \frac{2\mu}{p+2}\norm{\psi}_{L^{p+2}_x}^{p+2} \right) + 2\lambda\norm{B\psi}_{L^2_x}^2 = \lambda\int_{\T^2}\abs{u}^2\Re(\Bar{\psi}B\psi) + 2\lambda\int_{\T^2}u\cdot\Im(\nabla\Bar{\psi}B\psi).
	\end{equation} 
	The terms on the RHS are canceled once we include the energy of the normal fluid. We first rewrite ~\eqref{NSE} in the \textit{non-conservative form}, and apply the Leray projector (see Remark ~\ref{gradient terms in NSE}) to get
	\begin{equation}  \label{NSE'}
		\Leray\left(\rho\partial_t u + \rho u\cdot\nabla u - \nu \Delta u + \alpha\rho u\right) = \Leray\left(-2\lambda \Im(\nabla \Bar{\psi}B\psi) - 2\lambda u\Re(\Bar{\psi}B\psi) \right). \tag{NSE'}
	\end{equation} 
	Here, $\Leray$ is the Leray projector, which projects a Hilbert space into its divergence-free subspace, thus removing any purely gradient terms. We also apply the Leray projector to ~\eqref{NSE} to obtain
	\begin{equation}  \label{NSE-L}
		\Leray\left(\partial_t (\rho u) + \nabla\cdot(\rho u\otimes u) - \nu \Delta u + \alpha\rho u \right) = \Leray\left(-2\lambda \Im(\nabla \Bar{\psi}B\psi) \right). \tag{NSE-L}
	\end{equation} 
	Taking the inner product of both ~\eqref{NSE'} and ~\eqref{NSE-L} with $u$, using incompressibility, and adding them, we arrive at the energy equation for the normal fluid,
	\begin{equation} \label{normal fluid energy equation}
		\frac{1}{2}\frac{d}{dt}\norm{\sqrt{\rho}u}_{L^2_x}^2 + \nu\norm{\nabla u}_{L^2_x}^2 + \alpha\norm{\sqrt{\rho} u}_{L^2_x}^2 = - 2\lambda\int_{\T^2}u\cdot\Im(\nabla \Bar{\psi}B\psi) - \lambda\int_{\T^2}\abs{u}^2\Re(\Bar{\psi}B\psi).
	\end{equation}
	Therefore, by adding ~\eqref{superfluid energy equation} and ~\eqref{normal fluid energy equation}, we obtain the energy equation
	\begin{equation} \label{Energy equation}
		\frac{dE}{dt} + \nu\norm{\nabla u}_{L^2_x}^2 + \alpha\norm{\sqrt{\rho} u}_{L^2_x}^2 + 2\lambda\norm{B\psi}_{L^2_x}^2 = 0.
	\end{equation} 
	Thus, the energy is bounded from above as
	\begin{equation} \label{energy bound E0}
		E(t) + \nu\norm{\nabla u}_{L^2_{[0,T]}L^2_x}^2 + \alpha\norm{\sqrt{\rho}u}_{L^2_{[0,T]}L^2_x}^2 + 2\lambda\norm{B\psi}_{L^2_{[0,T]}L^2_x}^2 \\ = E_0 \quad , \quad t\in [0,T],
	\end{equation}  
	with
	\begin{equation} \label{E0 definition}
		E_0 := \frac{1}{2}\norm{\sqrt{\rho_0}u_0}_{L^2_x}^2 + \frac{1}{2}\norm{\nabla\psi_0}_{L^2_x}^2 + \frac{2\mu}{p+2}\norm{\psi_0}_{L^{p+2}_x}^{p+2}
	\end{equation}
	denoting the initial energy of the system.
	Next, we wish to show that the energy actually decays algebraically in time, under a certain smallness condition on the initial data. First, note that
	\begin{align*}
		\int_{\T^2} \abs{\psi}^p \Re (\overline{\psi}B_L\psi)
		&= \Re\int_{\T^2} \abs{\psi}^p \overline{\psi}\left[-\frac{1}{2}\Delta\psi + \frac{1}{2}\abs{u}^2 \psi + iu\cdot\nabla\psi \right] \\
		&= \frac{1}{2}\int_{\T^2} \abs{\psi}^p \abs{\nabla\psi}^2 + \frac{1}{4} \int_{\T^2} \nabla \abs{\psi}^p \cdot \nabla\abs{\psi}^2 + \frac{1}{2}\int_{\T^2}\abs{u}^2\abs{\psi}^{p+2} \\
		&\quad - \int_{\T^2} \abs{\psi}^p u\cdot \Im(\overline{\psi}\nabla\psi) \\
		&= \frac{1}{2}\int_{\T^2} \abs{\psi}^p \abs{\nabla\psi}^2 + \frac{2p}{(p+2)^2} \norm{\nabla \left( \abs{\psi}^{\frac{p}{2}+1} \right)}_{L^2_x}^2 + \frac{1}{2}\int_{\T^2}\abs{u}^2\abs{\psi}^{p+2} \\
		&\quad - \int_{\T^2} \abs{\psi}^p u\cdot \Im(\overline{\psi}\nabla\psi) \\
		&\gtrsim \norm{\nabla \left( \abs{\psi}^{\frac{p}{2}+1} \right)}_{L^2_x}^2 ,
	\end{align*}  
	where we used an argument similar to the one from the proof of Lemma ~\ref{coupling symmetric and non-negative} to get the last inequality. We now use ~\eqref{defining B_L} to see that
	\begin{equation} \label{B psi is bounded from below by D^2 psi}
		\begin{aligned}
			\norm{B\psi}_{L^2_x}^2 &= \norm{B_L\psi}_{L^2_x}^2 + \mu^2\norm{\psi}_{L^{2p+2}_x}^{2p+2} + 2\mu \int_{\T^2} \abs{\psi}^p \Re (\overline{\psi}B_L\psi) \\
			&\ge \norm{B_L\psi}_{L^2_x}^2 + \mu^2\norm{\psi}_{L^{2p+2}_x}^{2p+2} + \frac{1}{C} \norm{\nabla \left( \abs{\psi}^{\frac{p}{2}+1} \right)}_{L^2_x}^2 \\
			&\ge \frac{1}{8}\norm{D^2\psi}_{L^2_x}^2 - C\norm{\abs{u}^2\psi}_{L^2_x}^2 - C\norm{u\cdot\nabla\psi}_{L^2_x}^2 + \frac{1}{C}\norm{\psi}_{L^{2p+2}_x}^{2p+2} + \frac{1}{C} \norm{\nabla \left( \abs{\psi}^{\frac{p}{2}+1} \right)}_{L^2_x}^2 .
		\end{aligned}
	\end{equation}  
	Combining ~\eqref{Energy equation} and ~\eqref{B psi is bounded from below by D^2 psi}, we get
	\begin{equation} \label{energy inequation 1}
		\begin{aligned}
			&\frac{dE}{dt} + \nu\norm{\nabla u}_{L^2_x}^2 + \alpha\norm{\sqrt{\rho} u}_{L^2_x}^2 + \frac{\lambda}{4}\norm{D^2\psi}_{L^2_x}^2 + \frac{1}{C}\norm{\psi}_{L^{2p+2}_x}^{2p+2} + \frac{1}{C}\norm{\nabla \left( \abs{\psi}^{\frac{p}{2}+1} \right)}_{L^2_x}^2 \\ 
			&\lesssim \norm{\abs{u}^2\psi}_{L^2_x}^2 + \norm{u\cdot\nabla\psi}_{L^2_x}^2 =: I_1 + I_2.
		\end{aligned}
	\end{equation}
	We then bound the first term on the RHS using H\"older inequality and Gagliardo-Nirenberg (GN) interpolation as
	\begin{equation} \label{I1 bound}
		I_1 \lesssim \norm{u}_{L^6_x}^4\norm{\psi}_{L^6_x}^2 \lesssim \norm{u}_{L^2_x}^{\frac{4}{3}}\norm{u}_{H^1_x}^{\frac{8}{3}}\norm{\psi}_{H^1_x}^2.
	\end{equation}
	For the second term in ~\eqref{energy inequation 1}, we interpolate the $L^3_x$ norm, while also applying the H\"older, Poincar\'e, and Young inequalities, as well as the GN interpolation inequality, to get
	\begin{equation} \label{I2 bound}
		\begin{aligned}
			I_2 &\lesssim \norm{u}_{L^6_x}^2\norm{\nabla\psi}_{L^3_x}^2 \lesssim \norm{u}_{L^2_x}^{\frac{2}{3}}\norm{u}_{H^1_x}^{\frac{4}{3}}\norm{\nabla\psi}_{L^2_x}\norm{\nabla\psi}_{L^6_x} \\
			&\lesssim \norm{u}_{L^2_x}^{\frac{2}{3}}\norm{u}_{H^1_x}^{\frac{4}{3}}\norm{\nabla\psi}_{L^2_x}^{\frac{4}{3}} \norm{D^2\psi}_{L^2_x}^{\frac{2}{3}} \le C_{\kappa}\norm{u}_{L^2_x}\norm{u}_{H^1_x}^2\norm{\nabla\psi}_{L^2_x}^2 + \kappa\norm{D^2\psi}_{L^2_x}^2 \\
			&\le C_{\kappa} E_0^{\frac{3}{2}} \norm{\sqrt{\rho}u}_{L^2_x}^2 + C_{\kappa} E_0^{\frac{3}{2}} \norm{\nabla u}_{L^2_x}^2 + \kappa\norm{D^2\psi}_{L^2_x}^2. 
		\end{aligned}
	\end{equation}
	For sufficiently small values of $\kappa$ and $E_0$, the RHS of ~\eqref{I2 bound} can be absorbed into the LHS of ~\eqref{energy inequation 1}. We also use the Poincar\'e inequality to convert the last term on the LHS of ~\eqref{energy inequation 1} into a coercive term for the internal energy term $\frac{2\mu}{p+2}\norm{\psi}_{L^{p+2}_x}^{p+2}$ in $E(t)$. To this end, we observe that
	\begin{equation} \label{poincare for potential energy}
		\begin{aligned}
			\norm{\psi}_{L^{p+2}_x}^{p+2} &\le \norm{\abs{\psi}^{\frac{p}{2}+1} - \fint_{\T^2} \abs{\psi}^{\frac{p}{2}+1}}_{L^2_x}^2 + \norm{\fint_{\T^2} \abs{\psi}^{\frac{p}{2}+1}}_{L^2_x}^2 \lesssim \norm{\nabla \left( \abs{\psi}^{\frac{p}{2}+1} \right)}_{L^2_x}^2 + \norm{\psi}_{L^{\frac{p}{2}+1}}^{p+2} \\ &\le C\norm{\nabla \left( \abs{\psi}^{\frac{p}{2}+1} \right)}_{L^2_x}^2 + \kappa \norm{\psi}_{L^{p+2}_x}^{p+2} + C_{\kappa} \norm{\psi}_{L^2_x}^{p+2}.
		\end{aligned}
	\end{equation}
	In the last inequality, we interpolated between the $L^{p+2}_x$ and $L^2_x$ norms, which may be done when $p>2$. By choosing $\kappa$ sufficiently small, we can absorb the second term on the RHS into the LHS. For $p\le 2$, we can simply replace $\norm{\psi}_{L^{\frac{p}{2}+1}}^{p+2}$ by $\norm{\psi}_{L^2_x}^{p+2}$ since we are on a finite-size domain. Thus, irrespective of the value of $p$, ~\eqref{energy inequation 1} becomes
	\begin{equation} \label{energy inequation 3}
		\begin{aligned}
			&\frac{dE}{dt} + \frac{1}{C}\norm{\nabla u}_{L^2_x}^2 + \frac{1}{C}\norm{\sqrt{\rho} u}_{L^2_x}^2 + \frac{1}{C}\norm{D^2\psi}_{L^2_x}^2 + \frac{1}{C}\norm{\psi}_{L^{p+2}_x}^{p+2} + \frac{1}{C}\norm{\psi}_{L^{2p+2}_x}^{2p+2} \\
			&\quad \le C\norm{\psi}_{L^2_x}^{p+2} +  C\norm{u}_{L^2_x}^{\frac{4}{3}}\norm{u}_{H^1_x}^{\frac{8}{3}} \norm{\psi}_{H^1_x}^2.
		\end{aligned}
	\end{equation}  
	While we have the required coercive terms on the LHS, we cannot yet obtain a decay estimate for $E(t)$, since the second term on the RHS is out of reach using $E$ only. In order to control it, we set up an analogous inequality for a higher-order energy.
	
	\subsection{Higher-order energy estimate} \label{higher order estimate}
	In this subsection, we obtain further bounds for $\psi$ and $u$, this time with one more derivative than the energy $E$.
	
	\subsubsection{The Schr\"odinger equation} \label{NLS higher order estimate}
	Similarly to the case of the energy equation, we act upon ~\eqref{NLS} with the Laplacian $-\Delta$, multiply by $-\Delta\Bar{\psi}$, take the real part and integrate over the domain to get
	\begin{equation} \label{schrodinger equation higher order first step}
		\begin{aligned}
			\frac{1}{2}\frac{d}{dt}\norm{\Delta\psi}_{L^2_x}^2 &= - \lambda\Re\int_{\T^2}(\Delta^2\Bar{\psi})B\psi + \mu\Im\int_{\T^2}(\Delta^2\Bar{\psi})\abs{\psi}^p\psi =: I_3 + I_4.
		\end{aligned}
	\end{equation} 
	Once again, the first term on the RHS of ~\eqref{NLS} vanishes due to the boundary conditions. We now estimate the terms on the RHS of ~\eqref{schrodinger equation higher order first step}. For the first term,
	\begin{align*}
		I_3 &= \lambda\Re\int_{\T^2}\nabla(\Delta\Bar{\psi})\cdot\nabla (B\psi) \\ 
		&= \lambda\Re\int_{\T^2}\nabla(\Delta\Bar{\psi})\cdot\nabla\left(-\frac{1}{2}\Delta\psi + \frac{1}{2}\abs{u}^2\psi + iu\cdot\nabla\psi +\mu\abs{\psi}^p\psi\right) \\
		&= -\frac{\lambda}{2}\norm{D^3\psi}_{L^2_x}^2 + \lambda\Re\int_{\T^2}\nabla(\Delta\Bar{\psi})\cdot\nabla\left(\frac{1}{2}\abs{u}^2\psi + iu\cdot\nabla\psi +\mu\abs{\psi}^p\psi\right) \\
		&\le -\frac{\lambda}{4}\norm{D^3\psi}_{L^2_x}^2 + C\norm{\nabla(\abs{u}^2\psi)}_{L^2_x}^2 + C\norm{\nabla(u\cdot\nabla\psi)}_{L^2_x}^2 + C\norm{\nabla(\abs{\psi}^p\psi)}_{L^2_x}^2 ,
	\end{align*}
	which gives a dissipative term for $\psi$. For the term $I_4$, we again integrate by parts, followed by H\"older's inequality to obtain
	\begin{equation*}
		I_4 = -\mu\Im\int_{\T^2}\nabla(\Delta\Bar{\psi})\cdot\nabla(\abs{\psi}^p\psi) \le \frac{\lambda}{8}\norm{D^3\psi}_{L^2_x}^2 + C\norm{\nabla(\abs{\psi}^p\psi)}_{L^2_x}^2.
	\end{equation*}
	Thus, ~\eqref{schrodinger equation higher order first step} becomes
	\begin{equation} \label{schrodinger equation higher order third step}
		\begin{aligned}
			\frac{d}{dt}\norm{\Delta\psi}_{L^2_x}^2 + \frac{1}{C}\norm{D^3\psi}_{L^2_x}^2 &\lesssim  \norm{\nabla\left(\abs{u}^2\psi \right)}_{L^2_x}^2 + \norm{\nabla(u\cdot\nabla\psi)}_{L^2_x}^2 + \norm{\nabla \left( \abs{\psi}^p\psi \right) }_{L^2_x}^2 \\
			&=: I_5 + I_6 + I_7.
		\end{aligned}
	\end{equation}
	The first of these terms is bounded as
	\begin{equation} \label{I5 estimate}
		\begin{aligned}
			I_5 &\lesssim \norm{u}_{L^6_x}^2\norm{\nabla u}_{L^6_x}^2\norm{\psi}_{L^6_x}^2 + \norm{u}_{L^6_x}^4\norm{\nabla\psi}_{L^6_x}^2 \\
			&\lesssim \norm{u}_{L^2_x}^{\frac{2}{3}}\norm{u}_{H^1_x}^{\frac{4}{3}}\norm{\nabla u}_{L^2_x}^{\frac{2}{3}}\norm{\Delta u}_{L^2_x}^{\frac{4}{3}}\norm{\psi}_{H^1_x}^2 + \norm{u}_{L^2_x}^{\frac{4}{3}}\norm{u}_{H^1_x}^{\frac{8}{3}} \norm{\nabla\psi}_{L^2_x}^{\frac{2}{3}}\norm{\Delta\psi}_{L^2_x}^{\frac{4}{3}} \\
			&\le C_{\kappa} \norm{u}_{L^2_x}^2\norm{u}_{H^1_x}^4\norm{\nabla u}_{L^2_x}^2 \norm{\psi}_{H^1_x}^6 + \kappa\norm{\Delta u}_{L^2_x}^2 + C\norm{u}_{L^2_x}^{\frac{4}{3}}\norm{u}_{H^1_x}^{\frac{8}{3}} \norm{\nabla\psi}_{L^2_x}^{\frac{2}{3}}\norm{\Delta\psi}_{L^2_x}^{\frac{4}{3}} 
			,
		\end{aligned}
	\end{equation}
	using the Poincar\'e and GN interpolation inequalities. We have also applied Young's inequality to extract out dissipative terms in the last step. We again use $\kappa$ to denote a small number whose value shall be fixed later on, and $C_{\kappa}$ is a constant whose value depends on $\kappa$ and the system parameters. Similarly, for the second term on the RHS of ~\eqref{schrodinger equation higher order third step}, we have
	\begin{equation} \label{I6 estimate}
		\begin{aligned}
			I_6 &\lesssim \norm{\nabla u}_{L^3_x}^2 \norm{\nabla\psi}_{L^6_x}^2 + \norm{u}_{L^6_x}^2 \norm{D^2\psi}_{L^3_x}^2 \\
			&\lesssim \norm{\nabla u}_{L^2_x}^{\frac{4}{3}}\norm{D^2 u}_{L^2_x}^{\frac{2}{3}}\norm{\nabla\psi}_{L^2_x}^{\frac{2}{3}}\norm{D^2\psi}_{L^2_x}^{\frac{4}{3}} + \norm{u}_{L^2_x}^{\frac{2}{3}}\norm{u}_{H^1_x}^{\frac{4}{3}} \norm{\Delta\psi}_{L^2_x}^{\frac{4}{3}} \norm{D^3\psi}_{L^2_x}^{\frac{2}{3}} \\
			&\le C_{\kappa}\norm{\nabla u}_{L^2_x}^2\norm{\nabla\psi}_{L^2_x} \norm{\Delta\psi}_{L^2_x}^2 + \kappa\norm{\Delta u}_{L^2_x}^2 + C_{\kappa}\norm{u}_{L^2_x}\norm{u}_{H^1_x}^2 \norm{\Delta\psi}_{L^2_x}^2 + \kappa\norm{D^3\psi}_{L^2_x}^2 .
		\end{aligned}
	\end{equation}
	Finally, we apply the Sobolev embedding and Poincar\'e inequalities to bound $I_7$. This leads to
	\begin{equation} \label{I7 estimate}
		I_7 \lesssim \norm{\psi}_{L^{2(p+1)}_x}^{2p} \norm{\nabla\psi}_{L^{2(p+1)}_x}^2 \lesssim \norm{\psi}_{H^1_x}^{2p} \norm{\nabla\psi}_{H^1_x}^2 \lesssim \norm{\psi}_{H^1_x}^{2p} \norm{\Delta\psi}_{L^2_x}^2.
	\end{equation}   
	Combining all these inequalities into ~\eqref{schrodinger equation higher order third step} results in
	\begin{equation} \label{schrodinger equation higher order fourth step}
		\begin{aligned}
			&\frac{d}{dt}\norm{\Delta\psi}_{L^2_x}^2 + \frac{1}{C}\norm{D^3\psi}_{L^2_x}^2 \\ 
			&\le C_{\kappa}\left(\norm{u}_{L^2_x}^2\norm{u}_{H^1_x}^4\norm{\nabla u}_{L^2_x}^2\norm{\psi}_{H^1_x}^6 + \norm{\nabla u}_{L^2_x}^2\norm{\nabla\psi}_{L^2_x}\norm{\Delta\psi}_{L^2_x}^2 + \norm{u}_{L^2_x}\norm{u}_{H^1_x}^2 \norm{\Delta\psi}_{L^2_x}^2 \right) \\ 
			&\quad + C\norm{u}_{L^2_x}^{\frac{4}{3}}\norm{u}_{H^1_x}^{\frac{8}{3}} \norm{\nabla\psi}_{L^2_x}^{\frac{2}{3}}\norm{\Delta\psi}_{L^2_x}^{\frac{4}{3}} + C\norm{\psi}_{H^1_x}^{2p} \norm{\Delta\psi}_{L^2_x}^2 + \kappa\norm{\Delta u}_{L^2_x}^2 
			,
		\end{aligned}
	\end{equation}
	where we have absorbed $\kappa\norm{D^3\psi}_{L^2_x}^2$ into the LHS with a sufficiently small ~$\kappa$.
	
	\subsubsection{The Navier-Stokes equations} \label{NSE higher order estimate}
	We shall now derive a higher order estimate for the velocity field, which shall be combined with ~\eqref{schrodinger equation higher order fourth step}. Starting with ~\eqref{NSE'}, we first multiply it by $\partial_t u$ and integrate over the domain to obtain
	\begin{equation} \label{NSE higher order first step}
		\begin{aligned}
			\int_{\T^2}\rho\abs{\partial_t u}^2 + \frac{\nu}{2}\frac{d}{dt}\norm{\nabla u}_{L^2_x}^2 &= -\int_{\T^2}\rho (u\cdot\nabla u)\cdot\partial_t u - 2\lambda\int_{\T^2}\partial_t u\cdot\Im(\nabla\Bar{\psi}B\psi) \\ 
			&\qquad - 2\lambda\int_{\T^2}\partial_t u\cdot u\Re(\Bar{\psi}B\psi) -\alpha\int_{\T^2}\rho u\cdot\partial_t u \\
			&=: I_8 + I_9 + I_{10} + I_{11}.
		\end{aligned}
	\end{equation} 
	Recalling that $\mf\le \rho \le \Mf=\Mi+\mi-\mf$, we control the RHS. For the first term,  
	\begin{align*}
		I_8 &\le \frac{1}{8}\norm{\sqrt{\rho}\partial_t u}_{L^2_x}^2 + C \int_{\T^2}\abs{u}^2\abs{\nabla u}^2 \le \frac{1}{8}\norm{\sqrt{\rho}\partial_t u}_{L^2_x}^2 + C\norm{u}_{L^6_x}^2 \norm{\nabla u}_{L^3_x}^2 \\
		&\le \frac{1}{8}\norm{\sqrt{\rho}\partial_t u}_{L^2_x}^2 + C_{\kappa}\norm{u}_{L^2_x}\norm{u}_{H^1_x}^2 \norm{\nabla u}_{L^2_x}^2 + \kappa\norm{\Delta u}_{L^2_x}^2.
	\end{align*}
	In going to the last inequality, we used the GN interpolation and Poincar\'e inequalities. 
	Finally, Young's inequality lets us extract the required dissipative term. For the second integral in ~\eqref{NSE higher order first step},
	\begin{align*}
		I_9 &\le \frac{1}{8}\norm{\sqrt{\rho}\partial_t u}_{L^2_x}^2 + C\norm{\nabla\psi}_{L^6_x}^2\norm{B\psi}_{L^3_x}^2 \\
		&\le \frac{1}{8}\norm{\sqrt{\rho}\partial_t u}_{L^2_x}^2 + C\norm{\nabla\psi}_{L^2_x}^{\frac{2}{3}}\norm{\Delta\psi}_{L^2_x}^{\frac{4}{3}}\norm{B\psi}_{L^2_x}^{\frac{4}{3}}\norm{B\psi}_{H^1_x}^{\frac{2}{3}} \\
		&\le \frac{1}{8}\norm{\sqrt{\rho}\partial_t u}_{L^2_x}^2 + C_{\kappa}\norm{B\psi}_{L^2_x}^2 \left( \norm{\nabla\psi}_{L^2_x}^{\frac{2}{3}}\norm{\Delta\psi}_{L^2_x}^{\frac{4}{3}} + \norm{\nabla\psi}_{L^2_x}\norm{\Delta\psi}_{L^2_x}^2 \right) + \kappa \norm{\nabla(B\psi)}_{L^2_x}^2 ,
	\end{align*}  
	where the $B\psi$ term is handled via the GN interpolation and Young's inequalities. In the third integral in ~\eqref{NSE higher order first step},
	\begin{align*}
		I_{10} &\le \frac{1}{8}\norm{\sqrt{\rho}\partial_t u}_{L^2_x}^2 + C\norm{u}_{L^6_x}^2\norm{\psi}_{L^6_x}^2\norm{B\psi}_{L^6_x}^2 \\
		&\le \frac{1}{8}\norm{\sqrt{\rho}\partial_t u}_{L^2_x}^2 + C_{\kappa}\norm{B\psi}_{L^2_x}^2 \left( \norm{u}_{L^2_x}^{\frac{2}{3}}\norm{u}_{H^1_x}^{\frac{4}{3}}\norm{\psi}_{H^1_x}^2 + \norm{u}_{L^2_x}^2\norm{u}_{H^1_x}^4\norm{\psi}_{H^1_x}^6 \right) \\
		&\quad + \kappa\norm{\nabla(B\psi)}_{L^2_x}^2 ,
	\end{align*} 
	where the term $B\psi$ is handled just like in $I_9$. Finally, for the last term in ~\eqref{NSE higher order first step},
	\begin{equation}
		\begin{aligned} \label{HE9 3 terms}
			I_{11} &=  -\frac{\alpha}{2} \int_{\T^2} \rho \partial_t \abs{u}^2 = -\frac{\alpha}{2} \frac{d}{dt} \int_{\T^2} \rho \abs{u}^2 + \frac{\alpha}{2} \int_{\T^2} (\partial_t \rho ) \abs{u}^2 \\
			&= -\frac{\alpha}{2} \frac{d}{dt} \norm{ \sqrt{\rho} u}_{L^2_x}^2 - \frac{\alpha}{2} \int_{\T^2} \left( \nabla\cdot(\rho u) - 2\lambda\Re(\overline{\psi}B\psi) \right) \abs{u}^2\\
			&= -\frac{\alpha}{2} \frac{d}{dt} \norm{ \sqrt{\rho} u}_{L^2_x}^2 + \frac{\alpha}{2} \int_{\T^2} \rho u\cdot \nabla \abs{u}^2 + \alpha\lambda\int_{\T^2} \Re(\overline{\psi}B\psi) \abs{u}^2.
		\end{aligned}
	\end{equation}
	We estimate the second term on the RHS of ~\eqref{HE9 3 terms} using the H\"older and GN interpolation inequalities. This gives
	\begin{equation*}
		\begin{aligned}
			\frac{\alpha}{2} \int_{\T^2} \rho u\cdot \nabla \abs{u}^2 &\lesssim \norm{u}_{L^4_x}^2 \norm{\nabla u}_{L^2_x} \lesssim \norm{u}_{L^2_x}\norm{u}_{H^1_x}\norm{\nabla u}_{L^2_x}.
		\end{aligned}
	\end{equation*}
	Similarly, for the third term in ~\eqref{HE9 3 terms},
	\begin{align*}
		\alpha\lambda\int_{\T^2} \Re(\overline{\psi}B\psi) \abs{u}^2 &\lesssim \norm{\psi}_{L^6_x} \norm{u}_{L^6_x}^2 \norm{B\psi}_{L^2_x} \le C_{\kappa} \norm{\psi}_{H^1_x}^2\norm{u}_{L^2_x}^{\frac{4}{3}}\norm{u}_{H^1_x}^{\frac{8}{3}} + \kappa\norm{B\psi}_{L^2_x}^2.
	\end{align*}
	Substituting the above estimates into ~\eqref{NSE higher order first step}, we arrive at
	\begin{equation} \label{NSE higher order second step}
		\begin{aligned}
			&\nu\frac{d}{dt}\norm{\nabla u}_{L^2_x}^2 + \norm{\sqrt{\rho}\partial_t u}_{L^2_x}^2 + \alpha \frac{d}{dt} \norm{ \sqrt{\rho} u}_{L^2_x}^2 \\ 
			&\le C_{\kappa}\norm{B\psi}_{L^2_x}^2 \left(\norm{\nabla\psi}_{L^2_x}^{\frac{2}{3}}\norm{\Delta\psi}_{L^2_x}^{\frac{4}{3}} + \norm{\nabla\psi}_{L^2_x}\norm{\Delta\psi}_{L^2_x}^2 \right) \\
			&\quad + C_{\kappa}\norm{B\psi}_{L^2_x}^2 \left(\norm{u}_{L^2_x}^{\frac{2}{3}}\norm{u}_{H^1_x}^{\frac{4}{3}}\norm{\psi}_{H^1_x}^2 + \norm{u}_{L^2_x}^2\norm{u}_{H^1_x}^4\norm{\psi}_{H^1_x}^6 \right) \\
			&\quad + C_{\kappa} \left( \norm{u}_{L^2_x}\norm{u}_{H^1_x}^2 \norm{\nabla u}_{L^2_x}^2 + \norm{u}_{L^2_x}\norm{u}_{H^1_x} \norm{\nabla u}_{L^2_x} + \norm{\psi}_{H^1_x}^2\norm{u}_{L^2_x}^{\frac{4}{3}}\norm{u}_{H^1_x}^{\frac{8}{3}} \right) \\ 
			&\quad + \kappa\norm{B\psi}_{L^2_x}^2 + \kappa\norm{\nabla(B\psi)}_{L^2_x}^2 + \kappa\norm{\Delta u}_{L^2_x}^2 ,
		\end{aligned}
	\end{equation}
	where $C_{\kappa}$ depends on $\kappa$ and the system parameters.
	
	So far, we have obtained equations for $\norm{\nabla u}_{L^2_x}$ and $\norm{\Delta\psi}_{L^2_x}$, while including the higher-order dissipation corresponding to the wavefunction, $\norm{\nabla(B\psi)}_{L^2_x}^2$ or $\norm{D^3 \psi}_{L^2_x}^2$. What remains is to consider the higher-order velocity dissipation $\norm{\Delta u}_{L^2_x}^2$. To this end, we multiply ~\eqref{NSE'} by $-\theta\Delta u$, with $\theta$ to be determined, and integrate over the domain. This gives
	\begin{equation} \label{NSE higher order third step}
		\begin{aligned}
			\theta\nu\norm{\Delta u}_{L^2_x}^2 &= \theta\int_{\T^2}\rho\partial_t u\cdot\Delta u + \theta\int_{\T^2} \rho\left(u\cdot\nabla u \right)\cdot\Delta u + 2\lambda\theta\int_{\T^2}\Im(\nabla\Bar{\psi}B\psi)\cdot\Delta u \\ 
			&\qquad + 2\lambda\theta\int_{\T^2}\Re(\Bar{\psi}B\psi)u\cdot\Delta u + \alpha\theta\int_{\T^2} \rho u\cdot \Delta u \\
			&=: I_{12} + I_{13} + I_{14} + I_{15} + I_{16}.
		\end{aligned}
	\end{equation}
	When estimating the RHS, the goal is to extract $\norm{\Delta u}_{L^2_x}^2$ with a small coefficient, so it can be absorbed into the LHS. Thus, we have
	\begin{equation*}
		I_{12} \le \frac{\theta\nu}{10}\norm{\Delta u}_{L^2_x}^2 + C\theta\norm{\sqrt{\rho}\partial_t u}_{L^2_x}^2.
	\end{equation*} 
	The second integral is manipulated just as $I_8$ and yields
	\begin{align*}
		I_{13} &\le \frac{\theta\nu}{20}\norm{\Delta u}_{L^2_x}^2 + C_{\theta}\int_{\T^2}\abs{u}^2\abs{\nabla u}^2 \\
		&\le \frac{\theta\nu}{10}\norm{\Delta u}_{L^2_x}^2 + C_{\theta}\norm{u}_{L^2_x}\norm{u}_{H^1_x}^2 \norm{\nabla u}_{L^2_x}^2.
	\end{align*} 
	The bound for the integral $I_{14}$ follows from the GN interpolation, Poincar\'e, and Young inequalities, as
	\begin{align*}
		I_{14} &\le \frac{\theta\nu}{10}\norm{\Delta u}_{L^2_x}^2 + C_{\theta}\norm{\nabla\psi}_{L^6_x}^2\norm{B\psi}_{L^3_x}^2 \\
		&\le \frac{\theta\nu}{10}\norm{\Delta u}_{L^2_x}^2 + C_{\theta}\norm{\nabla\psi}_{L^2_x}^{\frac{2}{3}}\norm{\Delta\psi}_{L^2_x}^{\frac{4}{3}}\norm{B\psi}_{L^2_x}^{\frac{4}{3}} \norm{B\psi}_{H^1_x}^{\frac{2}{3}} \\
		&\le \frac{\theta\nu}{10}\norm{\Delta u}_{L^2_x}^2 + C_{\kappa,\theta}\norm{B\psi}_{L^2_x}^2 \left( \norm{\nabla\psi}_{L^2_x}^{\frac{2}{3}}\norm{\Delta\psi}_{L^2_x}^{\frac{4}{3}} + \norm{\nabla\psi}_{L^2_x}\norm{\Delta\psi}_{L^2_x}^2 \right) + \kappa\norm{\nabla(B\psi)}_{L^2_x}^2.
	\end{align*} 
	In a similar manner, we have
	\begin{align*}
		I_{15} &\le \frac{\theta\nu}{10}\norm{\Delta u}_{L^2_x}^2 + C_{\theta}\norm{u}_{L^6_x}^2\norm{\psi}_{L^6_x}^2\norm{B\psi}_{L^6_x}^2 \\
		&\le \frac{\theta\nu}{10}\norm{\Delta u}_{L^2_x}^2 + C_{\theta}\norm{u}_{L^2_x}^{\frac{2}{3}}\norm{u}_{H^1_x}^{\frac{4}{3}}\norm{\psi}_{H^1_x}^2\norm{B\psi}_{L^2_x}^{\frac{2}{3}}\norm{B\psi}_{H^1_x}^{\frac{4}{3}} \\
		&\le \frac{\theta\nu}{10}\norm{\Delta u}_{L^2_x}^2 + C_{\kappa,\theta}\norm{B\psi}_{L^2_x}^2 \left( \norm{u}_{L^2_x}^{\frac{2}{3}}\norm{u}_{H^1_x}^{\frac{4}{3}}\norm{\psi}_{H^1_x}^2 + \norm{u}_{L^2_x}^2\norm{u}_{H^1_x}^4\norm{\psi}_{H^1_x}^6 \right) + \kappa\norm{\nabla(B\psi)}_{L^2_x}^2.
	\end{align*} 
	Finally, for the last integral in ~\eqref{NSE higher order third step},
	\begin{equation*}
		I_{16} \le \frac{\theta\nu}{10}\norm{\Delta u}_{L^2_x}^2 + C\theta\norm{\sqrt{\rho}u}_{L^2_x}^2.
	\end{equation*}
	Thus, ~\eqref{NSE higher order third step} becomes
	\begin{equation} \label{NSE higher order fourth step}
		\begin{aligned}
			\frac{\theta\nu}{2}\norm{\Delta u}_{L^2_x}^2 &\le C_{\kappa,\theta} \norm{B\psi}_{L^2_x}^2 \left( \norm{\nabla\psi}_{L^2_x}^{\frac{2}{3}}\norm{\Delta\psi}_{L^2_x}^{\frac{4}{3}} + \norm{\nabla\psi}_{L^2_x}\norm{\Delta\psi}_{L^2_x}^2 \right) \\
			&\quad + C_{\kappa,\theta} \norm{B\psi}_{L^2_x}^2 \left( \norm{u}_{L^2_x}^{\frac{2}{3}}\norm{u}_{H^1_x}^{\frac{4}{3}}\norm{\psi}_{H^1_x}^2 + \norm{u}_{L^2_x}^2\norm{u}_{H^1_x}^4\norm{\psi}_{H^1_x}^6 \right) \\
			&\quad + C_{\theta}\norm{u}_{L^2_x}\norm{u}_{H^1_x}^2 \norm{\nabla u}_{L^2_x}^2 + C\theta\norm{\sqrt{\rho}\partial_t u}_{L^2_x}^2 + C\theta\norm{\sqrt{\rho}u}_{L^2_x}^2 + \kappa\norm{\nabla(B\psi)}_{L^2_x}^2 .
		\end{aligned}
	\end{equation} 
	
	We now add \eqref{schrodinger equation higher order fourth step}, \eqref{NSE higher order second step}, and~\eqref{NSE higher order fourth step}. We also observe that
	\begin{equation*}
		\norm{\nabla(B\psi)}_{L^2_x}^2 \lesssim \norm{D^3 \psi}_{L^2_x}^2 + \norm{\nabla\left(\abs{u}^2\psi \right)}_{L^2_x}^2 + \norm{\nabla(u\cdot\nabla\psi)}_{L^2_x}^2 + \norm{\nabla \left( \abs{\psi}^p\psi \right) }_{L^2_x}^2,
	\end{equation*}
	where the last three terms on the RHS are the same as $I_5,I_6,$ and $I_7$ in ~\eqref{schrodinger equation higher order third step}. We bound them just as in ~\eqref{I5 estimate}--\eqref{I7 estimate}. Choosing $\theta$ sufficiently small, and subsequently $\kappa$ also small enough, we absorb $\norm{\sqrt{\rho}\partial_t u}_{L^2_x}^2$ and $\norm{\Delta u}_{L^2_x}^2$ on the RHS into the corresponding terms on the LHS. Finally, what remains is
	\begin{equation} \label{delta psi nabla u equations combined}
		\begin{aligned}
			&\frac{d}{dt}\left( \norm{\Delta\psi}_{L^2_x}^2 + \nu\norm{\nabla u}_{L^2_x}^2 + \alpha\norm{\sqrt{\rho}u}_{L^2_x}^2 \right) + \frac{1}{C_{\kappa,\theta}}\norm{D^3\psi}_{L^2_x}^2 + \frac{1}{C_{\kappa,\theta}}\norm{\sqrt{\rho}\partial_t u}_{L^2_x}^2 + \frac{1}{C_{\kappa,\theta}}\norm{\Delta u}_{L^2_x}^2 \\
			&\le C_{\kappa,\theta} \left( \norm{u}_{L^2_x}\norm{u}_{H^1_x}^2 \norm{\nabla u}_{L^2_x}^2 + \norm{u}_{L^2_x}\norm{u}_{H^1_x} \norm{\nabla u}_{L^2_x} + \norm{\psi}_{H^1_x}^2\norm{u}_{L^2_x}^{\frac{4}{3}}\norm{u}_{H^1_x}^{\frac{8}{3}} \right) \\ 
			&\quad + C_{\kappa,\theta} \norm{B\psi}_{L^2_x}^2 \left( \norm{\nabla\psi}_{L^2_x}^{\frac{2}{3}}\norm{\Delta\psi}_{L^2_x}^{\frac{4}{3}} + \norm{\nabla\psi}_{L^2_x}\norm{\Delta\psi}_{L^2_x}^2 \right) \\
			&\quad + C_{\kappa,\theta} \norm{B\psi}_{L^2_x}^2 \left( \norm{u}_{L^2_x}^{\frac{2}{3}}\norm{u}_{H^1_x}^{\frac{4}{3}}\norm{\psi}_{H^1_x}^2 + \norm{u}_{L^2_x}^2\norm{u}_{H^1_x}^4\norm{\psi}_{H^1_x}^6 \right) \\
			&\quad + C_{\kappa}\left(\norm{u}_{L^2_x}^2\norm{u}_{H^1_x}^4\norm{\nabla u}_{L^2_x}^2\norm{\psi}_{H^1_x}^6 + \norm{\nabla u}_{L^2_x}^2\norm{\nabla\psi}_{L^2_x}\norm{\Delta\psi}_{L^2_x}^2 + \norm{u}_{L^2_x}\norm{u}_{H^1_x}^2 \norm{\Delta\psi}_{L^2_x}^2 \right) \\ 
			&\quad + C_{\kappa}\norm{u}_{L^2_x}^{\frac{4}{3}}\norm{u}_{H^1_x}^{\frac{8}{3}} \norm{\nabla\psi}_{L^2_x}^{\frac{2}{3}}\norm{\Delta\psi}_{L^2_x}^{\frac{4}{3}} + C_{\kappa}\norm{\psi}_{H^1_x}^{2p} \norm{\Delta\psi}_{L^2_x}^2 \\
			&\quad + C\theta\norm{\sqrt{\rho}u}_{L^2_x}^2 + \kappa\norm{B\psi}_{L^2_x}^2 ,
		\end{aligned}
	\end{equation}
	where we absorbed $\norm{D^3\psi}_{L^2_x}^2$ with an appropriate choice of ~$\kappa$. This is the higher-order energy estimate.
	
	\subsubsection{The Gr\"onwall inequality step} \label{gronwall inequality step for higher order estimate}
	Having derived the equations for the higher-order norms of $u$ and $\psi$, and while accounting for the relevant dissipative terms, the goal now is to use a Gr\"onwall-type argument. 
	
	\begin{lem}[Algebraic decay rate for energies] \label{lem:energy + higher order energy estimates}
		The sum of the energy $E(t)$ and the higher-order energy ~$X(t) := \norm{\Delta\psi(t)}_{L^2_x}^2 + \nu\norm{\nabla u(t)}_{L^2_x}^2$ decays algebraically in time as $(1+t)^{-\left(1+\frac{2}{p}\right)}$.
	\end{lem}
	
	\begin{proof}
		We begin by denoting 
		\begin{gather*}
			Y := \frac{1}{C} \left(\norm{D^3 \psi}_{L^2_x}^2 + \norm{\sqrt{\rho}\partial_t u}_{L^2_x}^2 + \norm{\Delta u}_{L^2_x}^2 \right) ,
		\end{gather*} 
		so we can rewrite ~\eqref{delta psi nabla u equations combined}, after updating $\theta,\kappa, E_0$, and $S_0$ to be sufficiently small, as
		\begin{equation} \label{gronwall expanded form}
			\begin{aligned}
				\frac{dX}{dt} + \alpha\frac{d}{dt} \norm{ \sqrt{\rho} u}_{L^2_x}^2 + Y &\le CS^p X + Q_1(X+E) + \norm{B\psi}_{L^2_x}^2  Q_2(X+E) \\ 
				&\qquad + \frac{\alpha}{2} \norm{\sqrt{\rho}u}_{L^2_x}^2 + \frac{\nu}{2}\norm{\nabla u}_{L^2_x}^2 + \lambda\norm{B\psi}_{L^2_x}^2 ,
			\end{aligned}
		\end{equation}
		where $Q_1(X+E)$ is a strictly super-linear polynomial, while $Q_2(X+E)$ contains both linear and super-linear terms. To arrive at ~\eqref{gronwall expanded form}, we have also expanded the Sobolev norms as
		\begin{equation} \label{u H1 norm expansion}
			\begin{multlined}
				\norm{u}_{H^1_x}^2 = \norm{u}_{L^2_x}^2 + \norm{\nabla u}_{L^2_x}^2
				\le \mf^{-1}\norm{\sqrt{\rho}u}_{L^2_x}^2 + \nu^{-1} (\nu\norm{\nabla u}_{L^2_x}^2)
				\lesssim E + X
			\end{multlined}
		\end{equation}
		for the velocity, and
		\begin{equation}\label{psi H2 norm expansion}
			\norm{\psi}_{H^1_x}^2 \le \norm{\psi}_{L^2_x}^2 + \norm{\nabla\psi}_{L^2_x}^2 \lesssim S + E 
		\end{equation}
		for the wavefunction.
		Next, we add ~\eqref{Energy equation} and ~\eqref{gronwall expanded form} to end up with
		\begin{equation} \label{XE equation initial}
			\begin{aligned}
				&\frac{d}{dt}(X+E) + Y + \frac{1}{C}\norm{\nabla u}_{L^2_x}^2 + \frac{1}{C}\norm{\sqrt{\rho}u}_{L^2_x}^2 + \frac{1}{C}\norm{B\psi}_{L^2_x}^2 \\
				&\le CS^p X + Q_1(X+E) + \norm{B\psi}_{L^2_x}^2 Q_2(X+E).
			\end{aligned}
		\end{equation} 
		We use the Poincar\'e inequality to rewrite $Y$ in order to get decaying norms. Indeed,
		\begin{equation*}
			Y \gtrsim \norm{\Delta\psi}_{L^2_x}^2 + \norm{\nabla u}_{L^2_x}^2 \gtrsim X.
		\end{equation*}
		Additionally, we also use the analysis in ~\eqref{B psi is bounded from below by D^2 psi}--\eqref{energy inequation 3} to rewrite $\norm{B\psi}_{L^2_x}^2$ on the LHS of ~\eqref{XE equation initial} in terms of $\norm{D^2\psi}_{L^2_x}^2$, which in turn can be downgraded to $\norm{\nabla\psi}_{L^2_x}^2$ using the Poincar\'e inequality. 
		One can also represent $\norm{B\psi}_{L^2_x}^2$ on the RHS of ~\eqref{XE equation initial} by
		\begin{align*}
			\norm{B\psi}_{L^2_x}^2 &\lesssim \norm{\Delta \psi}_{L^2_x}^2 + \norm{\abs{u}^2\psi}_{L^2_x}^2 + \norm{u\cdot\nabla\psi}_{L^2_x}^2 +  \norm{\psi}_{L^{2p+2}_x}^{2p+2} \\
			&\lesssim X + (S_0 + E_0) (X+E)^2 + (X+E)^2 + \norm{\psi}_{L^2_x}^2\norm{\psi}_{H^1_x}^{2p} \\
			&\lesssim X + (S_0 + E_0 + 1) (X+E)^2 + S_0^{p+1} + S_0 E^p ,
		\end{align*} 
		where we have used the estimates ~\eqref{I1 bound} and ~\eqref{I2 bound}, and the GN inequality. After all of the above manipulations, ~\eqref{XE equation initial} now reads
		\begin{equation} \label{XE equation coercive}
			\frac{d}{dt}(X+E) + \beta (X+E) 
			\le CS^{\frac{p}{2}+1} + CS_0^p (1+S_0) (X+E) + Q_1(X+E) + \Tilde{Q}_2(X+E) ,
		\end{equation} 
		where $\beta$ depends on the system parameters, and the polynomials $Q_1$ and $\Tilde{Q}_2$ are strictly super-linear. The first term on the RHS results from the estimates in ~\eqref{poincare for potential energy}. As for the second term on the RHS, we note that this can be absorbed into the LHS by tweaking $S_0$.
		
		For notational convenience, we write $Z := X+E$ and use $Q := Q_1 + \Tilde{Q}_2$ to denote the strictly super-linear polynomial in the RHS of ~\eqref{XE equation coercive}, leaving us with
		\begin{equation} \label{Z equation}
			\frac{dZ}{dt} + \beta Z \le \frac{C S_0^{\frac{p}{2}+1}}{\left( 1+S_0^{\frac{p}{2}}t \right)^{1+\frac{2}{p}}} + Q(Z).
		\end{equation} 
		The Duhamel solution for $Z(t)$ obeys
		\begin{equation} \label{Z duhamel solution}
			Z(t) \le e^{-\beta t}Z_0 + C S_0^{\frac{p}{2}+1}\int_0^t \frac{e^{-\beta (t-s)}}{\left( 1+S_0^{\frac{p}{2}}s \right)^{1+\frac{2}{p}}} ds + \int_0^t e^{-\beta (t-s)} Q\left(Z(s) \right) ds.
		\end{equation}
		We set the size of the initial data as $Z(0) =: Z_0 \le \veps$.
		We need to use a bootstrap argument to show that $Z(t) \le 3\veps \text{ for } t\in [0,T]$. Specifically, we prove that the hypothesis $Z(t) \le 4\veps$ for $t\le t_1$ leads to the stronger conclusion $Z(t) \le 3\veps$ for $t\le t_1$, where $t_1\in [0,T]$.  To this end, we estimate each integral on the RHS of ~\eqref{Z duhamel solution}. The first integral is split into two parts to take advantage of the exponential decay factor. Thus,
		\begin{equation} \label{splitting time integral}
			\begin{aligned}
				\int_0^t \frac{e^{-\beta (t-s)}}{\left( 1+S_0^{\frac{p}{2}}s \right)^{1+\frac{2}{p}}} ds &= \int_0^{\frac{t}{2}} \frac{e^{-\beta (t-s)}}{\left( 1+S_0^{\frac{p}{2}}s \right)^{1+\frac{2}{p}}} ds + \int_{\frac{t}{2}}^t \frac{e^{-\beta (t-s)}}{\left( 1+S_0^{\frac{p}{2}}s \right)^{1+\frac{2}{p}}} ds \\
				&\le e^{-\beta (t-\frac{t}{2})}\int_0^{\frac{t}{2}} ds + \frac{1}{\left( 1+S_0^{\frac{p}{2}}\frac{t}{2} \right)^{1+\frac{2}{p}}} \int_{\frac{t}{2}}^t e^{-\beta (t-s)} ds \\
				&\le \frac{t}{2} e^{-\beta \frac{t}{2}} + \frac{\beta^{-1}}{\left( 1+S_0^{\frac{p}{2}}\frac{t}{2} \right)^{1+\frac{2}{p}}} \le \frac{C}{\left( 1+S_0^{\frac{p}{2}}t \right)^{1+\frac{2}{p}}}.
			\end{aligned}
		\end{equation} 
		The last inequality is a result of the exponential decay of the first term, compared to the algebraic decay of the second. The second integral in ~\eqref{Z duhamel solution} is more straightforward, with
		\begin{equation*}
			\int_0^t e^{-\beta (t-s)} Q\left(Z(s)\right) ds \le Q\left(4\veps \right) \int_0^t e^{-\beta (t-s)} ds \le C Q\left(4\veps \right).
		\end{equation*} 
		Now we choose $\veps$ small enough (call it $\veps_0$) so that the RHS $\le \veps$. Similarly, the contribution from the first integral term in ~\eqref{Z duhamel solution} is made less than $\veps$ for all $S_0 \le \veps_0<1$ small enough.
		This completes the bootstrap argument, and we can see that indeed $Z(t)\le 3\veps$. For $\veps_0$ sufficiently small, the linear dissipation in ~\eqref{Z equation} dominates the nonlinearities, and we may write the equation as
		\begin{equation} \label{Z equation nonlinearities absorbed}
			\frac{dZ}{dt} + \frac{1}{C} Z \le \frac{CS_0^{\frac{p}{2}+1}}{\left( 1+S_0^{\frac{p}{2}}t \right)^{1+\frac{2}{p}}} ,
		\end{equation} 
		whose solution, following ~\eqref{Z equation}--\eqref{splitting time integral}, obeys
		\begin{equation} \label{Z solution bound}
			Z(t) \le Z_0 e^{-\frac{t}{C}} + \frac{CS_0^{\frac{p}{2}+1}}{\left( 1+S_0^{\frac{p}{2}}t \right)^{1+\frac{2}{p}}} \lesssim Z_0 + S_0^{\frac{p}{2}+1} .
		\end{equation}
		Returning to ~\eqref{XE equation initial}, we now absorb the last term on the RHS into the LHS, which is possible for small enough data since $\sup_{0\le t\le T} Z(t)\le 3\veps_0$. Furthermore, in the regime of small data, the super-linear polynomial $Q_1$ can be dominated by the linear term on the RHS, which leaves us with
		\begin{equation} \label{XE equation final}
			\frac{dZ}{dt} + Y + \frac{1}{C}\norm{\nabla u}_{L^2_x}^2 + \frac{1}{C}\norm{\sqrt{\rho}u}_{L^2_x}^2 + \frac{1}{C}\norm{B\psi}_{L^2_x}^2 \le CS_0^p Z.
		\end{equation}
		Employing the bound for $Z$ from ~\eqref{Z solution bound} in the RHS of ~\eqref{XE equation final} and integrating over $[0,T]$, we estimate the dissipation as
		\begin{equation} \label{Y dissipative estimates}
			\begin{aligned} 
				&\norm{D^3 \psi}_{L^2_{[0,T]} L^2_x}^2 + \norm{\sqrt{\rho}\partial_t u}_{L^2_{[0,T]} L^2_x}^2 + \norm{\Delta u}_{L^2_{[0,T]} L^2_x}^2 + \norm{\nabla u}_{L^2_{[0,T]} L^2_x}^2 + \norm{\sqrt{\rho}u}_{L^2_{[0,T]} L^2_x}^2 + \norm{B\psi}_{L^2_{[0,T]}
					L^2_x}^2 \\ 
				&\lesssim Z_0 + Z_0 S_0^p + S_0^{p+1} \lesssim Z_0 + S_0^{p+1} .
			\end{aligned}
		\end{equation} 
		The last inequality holds because $S_0<1$. Thus, we can achieve small values for the RHS of ~\eqref{Y dissipative estimates} by selecting appropriate $Z_0$ and $S_0$.
		
		Another useful estimate for the dissipative terms results from integrating ~\eqref{XE equation final} over the time interval $[t,2t]$, where $t\ge 1$. This gives
		\begin{equation} \label{improved Y dissipative estimates}
			\begin{aligned} 
				&\norm{D^3 \psi}_{L^2_{[t,2t]} L^2_x}^2 + \norm{\sqrt{\rho}\partial_t u}_{L^2_{[t,2t]} L^2_x}^2 + \norm{\Delta u}_{L^2_{[t,2t]} L^2_x}^2 + \norm{\nabla u}_{L^2_{[t,2t]} L^2_x}^2 + \norm{\sqrt{\rho}u}_{L^2_{[t,2t]} L^2_x}^2 + \norm{B\psi}_{L^2_{[t,2t]} L^2_x}^2 \\ 
				&\lesssim Z_0 e^{-\frac{t}{C}} + \frac{S_0^{\frac{p}{2}+1}}{\left(1+S_0^{\frac{p}{2}}t\right)^{\frac{2}{p}}}. 
			\end{aligned}
		\end{equation}
		This time-decaying bound on the dissipation is necessary to obtain a sharp control of the dynamics at large times.
	\end{proof} 
	
	\subsection{The highest-order a priori estimate for $\psi$} \label{the highest-order a priori estimate for psi}
	
	From the previous analysis, we have obtained $B\psi\in L^2_{[0,T]}H^1_x$. However, as pointed out in the discussion following Definition ~\ref{existence time definition}, we seek $B\psi\in ~L^2_{[0,T]}L^{\infty}_x$. To this end, we would like to obtain an even higher order a priori estimate, only for $\psi$.
	
	\begin{lem}[Algebraic decay rate for highest-order norm of $\psi$] \label{lem:highest order energy estimates}
		For $S_0, E_0$, and $Z_0$ small enough, and with $s=\frac{5}{4}$, the homogeneous Sobolev norm $\norm{\psi(t)}_{\dot{H}^{2s}_x}$ decays algebraically in time as $(1+t)^{-2-\frac{2}{p}}$. In addition, if $\norm{\psi_0}_{\dot{H}^{2s}_x}$ is sufficiently small, the higher-order dissipation $\norm{\psi}_{L^2_{[0,T]} \dot{H}^{2s+1}_x}$ may be made as small as required, independent of the time ~$T$.
	\end{lem}
	The choice of $s = \frac{5}{4}$ is dictated by $u\in L^{\infty}_t H^{1}_x \cap L^2_t H^{2}_x \subset L^4_t H^{\frac{3}{2}}_x$. This inclusion is precisely what we need to control the term $\abs{u}^2 \psi$ in the coupling using the a priori estimates up to this point.
	\begin{proof}
		With $s=\frac{5}{4}$, we apply $(-\Delta)^s$ to ~\eqref{NLS}, to get
		\begin{equation} \label{highest order NLS}
			\partial_t (-\Delta)^s\psi + \lambda(-\Delta)^s(B\psi) = -\frac{1}{2i}\Delta(-\Delta)^s\psi + \frac{\mu}{i}(-\Delta)^s(\abs{\psi}^p \psi).
		\end{equation} 
		Just as in Sections ~\ref{energy estimate} and ~\ref{NLS higher order estimate}, we multiply by $(-\Delta)^s\Bar{\psi}$, and integrate the real part over $\T^2$. As a result, the second term on the LHS yields
		\begin{equation} \label{B psi highest order estimate preliminary}
			\begin{aligned}
				&\Re\int_{\T^2} (-\Delta)^s \Bar{\psi} (-\Delta)^s (B\psi) = \Re\left\langle (-\Delta)^s \psi, (-\Delta)^s (B\psi) \right\rangle \\
				&= \Re\left\langle (-\Delta)^{s+\frac{1}{2}} \psi, (-\Delta)^{s-\frac{1}{2}} (B\psi) \right\rangle = \Re\left\langle (-\Delta)^{s-\frac{1}{2}} (-\Delta)\psi, (-\Delta)^{s-\frac{1}{2}} (B\psi) \right\rangle.
			\end{aligned}
		\end{equation} 
		Using the self-adjoint property of the Laplacian allows us to conclude that the first term on the RHS of ~\eqref{highest order NLS} vanishes, since
		\begin{equation*}
			\Im\int_{\T^2} (-\Delta)^s \Bar{\psi} (-\Delta)^{s+1} \psi = \Im \left\langle (-\Delta)^s \psi, (-\Delta)^{s+1}\psi \right\rangle = \Im \left\langle (-\Delta)^{s+\frac{1}{2}}\psi, (-\Delta)^{s+\frac{1}{2}}\psi \right\rangle = 0.
		\end{equation*}
		For the second term on the RHS of ~\eqref{highest order NLS}, we have
		\begin{equation*}
			\begin{aligned}
				&\Im\int_{\T^2} (-\Delta)^s \Bar{\psi} (-\Delta)^{s} (\abs{\psi}^p\psi) = \left\langle (-\Delta)^s\psi,(-\Delta)^{s} (\abs{\psi}^p\psi) \right\rangle \\
				&= \left\langle (-\Delta)^{s+\frac{1}{2}}\psi,(-\Delta)^{s-\frac{1}{2}} (\abs{\psi}^p\psi) \right\rangle = \left\langle (-\Delta)^{s-\frac{1}{2}}(-\Delta)\psi,(-\Delta)^{s-\frac{1}{2}} (\abs{\psi}^p\psi) \right\rangle.
			\end{aligned}
		\end{equation*} 
		In the last expression of ~\eqref{B psi highest order estimate preliminary}, we expand the operator $B$ and use H\"older's inequality to arrive at
		\begin{equation} \label{highest order energy inequality}
			\begin{aligned}
				&\frac{d}{dt}\norm{(-\Delta)^s\psi}_{L^2_x}^2 + \frac{1}{C}\norm{(-\Delta)^{s-\frac{1}{2}}(-\Delta\psi)}_{L^2_x}^2 \\ 
				&\lesssim  \norm{(-\Delta)^{s-\frac{1}{2}}(\abs{u}^2\psi)}_{L^2_x}^2 + \norm{(-\Delta)^{s-\frac{1}{2}}(u\cdot\nabla\psi)}_{L^2_x}^2 + \norm{(-\Delta)^{s-\frac{1}{2}}(\abs{\psi}^p\psi)}_{L^2_x}^2.
			\end{aligned}
		\end{equation}
		Rewriting the LHS in terms of the homogeneous Sobolev norms and the RHS in terms of the usual Sobolev norms, we have
		\begin{equation} \label{highest order energy sobolev norms}
			\frac{d}{dt}\norm{\psi}_{\dot{H}^{2s}_x}^2 + \frac{1}{C}\norm{\psi}_{\dot{H}^{2s+1}_x}^2 \lesssim  \norm{\abs{u}^2\psi}_{H^{2s-1}_x}^2 + \norm{u\cdot\nabla\psi}_{H^{2s-1}_x}^2 + \norm{\abs{\psi}^p\psi}_{H^{2s-1}_x}^2.
		\end{equation}
		Since $2s-1 = \frac{3}{2}$, the algebra property of Sobolev norms is applicable. Using this, ~\eqref{superfluid mass bound} and ~\eqref{Z solution bound}, 
		we estimate the RHS of ~\eqref{highest order energy sobolev norms}. The first term requires interpolation\footnote{In \cite{Jayanti2022LocalSuperfluidity}, the high norm of the velocity was estimated using the Lions-Magenes lemma, which cannot be applied.} and yields
		\begin{equation} \label{u^2 psi in highest norm}
			\begin{aligned}
				\norm{\abs{u}^2\psi}_{H^{2s-1}_x}^2 &\lesssim \norm{u}_{H^{\frac{3}{2}}_x}^4 \norm{\psi}_{H^{\frac{3}{2}}_x}^2 \lesssim \norm{u}_{H^1_x}^2 \norm{u}_{H^2_x}^2 \norm{\psi}_{H^2_x}^2 \lesssim Z(t) \Big(S(t) + Z(t) \Big) \norm{u}_{H^2_x}^2 \\
				&\lesssim \left( Z_0 e^{-\frac{t}{C}} + \frac{S_0^{\frac{p}{2}+2}}{\left(1+S_0^{\frac{p}{2}}t\right)^{1+\frac{4}{p}}} \right) \norm{u}_{H^2_x}^2
				,
			\end{aligned}
		\end{equation}
		where we have retained only the terms that decay the slowest. In arriving at the last inequality in ~\eqref{u^2 psi in highest norm}, we use the fact that $Z_0<1$. For the second term in the RHS of ~\eqref{highest order energy sobolev norms}, we similarly obtain 
		\begin{equation} \label{u.grad psi in highest norm}
			\norm{u\cdot\nabla\psi}_{H^{2s-1}_x}^2 \lesssim \norm{u}_{H^{\frac{3}{2}}_x}^2 \norm{\psi}_{\dot{H}^{2s}_x}^2 \lesssim \norm{u}_{H^2_x}^2 \norm{\psi}_{\dot{H}^{2s}_x}^2.
		\end{equation}
		While the $H^{\frac{3}{2}}_x$ norm could have been interpolated between $H^1_x$ and $H^2_x$, it does not provide an improved estimate since $\norm{u}_{L^2_t H^1_x}$ and $\norm{u}_{L^2_t H^2_x}$ are both bounded by ~\eqref{Y dissipative estimates}. In the last term of ~\eqref{highest order energy sobolev norms}, in view of Remark~\ref{restricting p>=1}, we have
		\begin{equation} \label{psi^p psi in highest norm}
			\begin{aligned}
				\norm{\abs{\psi}^p\psi}_{H^{2s-1}_x}^2 &\lesssim \norm{\psi}_{H^2_x}^{2p+2} \lesssim \left( \norm{\psi}_{L^2_x}^2 + \norm{\Delta\psi}_{L^2_x}^2 \right)^{p+1} \\
				&\lesssim \frac{S_0^{p+1}}{\left(1+S_0^{\frac{p}{2}} t\right)^{2+\frac{2}{p}}} + Z_0^{p+1} e^{-(p+1)\frac{t}{C}} + \frac{S_0^{(\frac{p}{2}+1)(p+1)}}{\left(1+S_0^{\frac{p}{2}} t\right)^{3+\frac{2}{p}+p}} \\ 
				&\lesssim Z_0 e^{-\frac{t}{C}} + \frac{S_0^{p+1}}{\left(1+S_0^{\frac{p}{2}} t\right)^{2+\frac{2}{p}}},
			\end{aligned}
		\end{equation} 
		where the penultimate inequality is obtained using ~\eqref{superfluid mass bound} and ~\eqref{Z solution bound}. Therefore, ~\eqref{highest order energy sobolev norms} becomes
		\begin{equation} \label{highest order energy LHS inhomogenous}
			\begin{aligned}
				\frac{d}{dt}\norm{\psi}_{\dot{H}^{2s}_x}^2 + \frac{1}{C}\norm{\psi}_{\dot{H}^{2s+1}_x}^2 &\lesssim  \norm{u}_{H^2_x}^2\norm{\psi}_{\dot{H}^{2s}_x}^2 + \left( Z_0 e^{-\frac{t}{C}} + \frac{S_0^{\frac{p}{2}+2}}{\left(1+S_0^{\frac{p}{2}}t\right)^{1+\frac{4}{p}}} \right) \norm{u}_{H^2_x}^2 \\ 
				&\quad + Z_0 e^{-\frac{t}{C}} + \frac{S_0^{p+1}}{\left(1+S_0^{\frac{p}{2}} t\right)^{2+\frac{2}{p}}}.
			\end{aligned}
		\end{equation}
		With the Poincar\'e inequality, we replace the dissipative term on the LHS with $W(t) := \norm{\psi(t)}_{\dot{H}^{2s}_x}^2$, which yields 
		\begin{equation} \label{highest order energy poincare gronwall}
			\begin{aligned}
				\frac{dW}{dt} + \frac{1}{C}W &\lesssim  \norm{u}_{H^2_x}^2 W + \left( Z_0 e^{-\frac{t}{C}} + \frac{S_0^{\frac{p}{2}+2}}{\left(1+S_0^{\frac{p}{2}}t\right)^{1+\frac{4}{p}}} \right) \norm{u}_{H^2_x}^2 \\ 
				&\quad + Z_0 e^{-\frac{t}{C}} + \frac{S_0^{p+1}}{\left(1+S_0^{\frac{p}{2}} t\right)^{2+\frac{2}{p}}} ,
			\end{aligned}
		\end{equation}
		whose solution obeys, using the Gr\"onwall inequality, 
		\begin{equation*} \label{homogeneous high norm bound 0}
			\begin{aligned}
				W(t) &\le e^{C\norm{u}_{L^2_{[0,T]} H^2_x}^2} W_0 e^{-\frac{t}{C}} \\ 
				&\quad + Ce^{C\norm{u}_{L^2_{[0,T]} H^2_x}^2} \int_0^t e^{-\frac{1}{C}(t-s)} \left( Z_0 e^{-\frac{s}{C}} + \frac{S_0^{\frac{p}{2}+2}}{\left(1+S_0^{\frac{p}{2}}s\right)^{1+\frac{4}{p}}} \right) \norm{u}_{H^2_x}^2 ds \\
				&\quad + Ce^{C\norm{u}_{L^2_{[0,T]} H^2_x}^2} \int_0^t e^{-\frac{1}{C}(t-s)} \left( Z_0 e^{-\frac{s}{C}} + \frac{S_0^{p+1}}{\left(1+S_0^{\frac{p}{2}} s\right)^{2+\frac{2}{p}}} \right) ds ,
			\end{aligned}
		\end{equation*}
		where $W_0 := \norm{\psi_0}_{\dot{H}^{2s}_x}^2$. We employ calculations similar to ~\eqref{splitting time integral} to estimate the integrals, i.e., splitting them over $\left[0,\frac{t}{2}\right]$ and $\left[\frac{t}{2},t\right]$. We also use ~\eqref{Y dissipative estimates}, in particular $\norm{u}_{L^2_{[0,T]} H^2_x}^2 \lesssim Z_0 + S_0^{p+1} \lesssim 1$, to simplify the exponential factors outside the integrals. In all, we end up with
		\begin{equation} \label{homogeneous high norm bound}
			\begin{aligned}
				W(t) &\lesssim W_0 e^{-\frac{t}{C}} + Z_0 \norm{u}_{L^2_{[0,t]} H^2_x}^2 e^{-\frac{t}{C}} + S_0^{\frac{p}{2}+2} \norm{u}_{L^2_{[0,\frac{t}{2}]} H^2_x}^2 e^{-\frac{t}{2C}} + \frac{S_0^{\frac{p}{2}+2}}{\left(1+S_0^{\frac{p}{2}}t\right)^{1+\frac{4}{p}}} \norm{u}_{L^2_{[\frac{t}{2},t]} H^2_x}^2 \\
				&\quad + Z_0 e^{-\frac{t}{2C}} + S_0^{p+1} \frac{t}{2} e^{-\frac{t}{2C}} + \frac{S_0^{p+1}}{\left(1+S_0^{\frac{p}{2}} t\right)^{2+\frac{2}{p}}} ,
			\end{aligned}
		\end{equation}
		for all $t\in[0,T]$. We simplify further by making use of ~\eqref{Y dissipative estimates} and ~\eqref{improved Y dissipative estimates} for $\norm{u}_{L^2_{[0,\frac{t}{2}]} H^2_x}^2$ and $\norm{u}_{L^2_{[\frac{t}{2},t]} H^2_x}^2$, respectively. This leads to
		\begin{equation} \label{homogeneous high norm bound 2}
			\begin{aligned}	
				\norm{\psi(t)}_{\dot{H}^{2s}_x}^2 &\lesssim W_0 e^{-\frac{t}{C}} + \left(Z_0 + S_0^{\frac{p}{2}+2}\right) \left(Z_0 + S_0^{p+1}\right) e^{-\frac{t}{2C}} \\ 
				&\quad + \frac{S_0^{\frac{p}{2}+2}}{\left(1+S_0^{\frac{p}{2}}t\right)^{1+\frac{4}{p}}} \left( Z_0 e^{-\frac{t}{C}} + \frac{S_0^{\frac{p}{2}+1}}{\left(1+S_0^{\frac{p}{2}}t \right)^{\frac{2}{p}}} \right) \\
				&\quad + Z_0 e^{-\frac{t}{2C}} + \frac{S_0^{p+1}}{\left(1+S_0^{\frac{p}{2}} t\right)^{2+\frac{2}{p}}} \\
				&\lesssim \left(W_0 + Z_0 + S_0^{\frac{3p}{2}+3} \right) e^{-\frac{t}{2C}} + \frac{S_0^{p+3}}{\left(1+S_0^{\frac{p}{2}} t\right)^{1+\frac{6}{p}}} + \frac{S_0^{p+1}}{\left(1+S_0^{\frac{p}{2}} t\right)^{2+\frac{2}{p}}} .
			\end{aligned}
		\end{equation}
		We use ~\eqref{homogeneous high norm bound 2} in ~\eqref{highest order energy LHS inhomogenous} and integrate over $[0,T]$ to obtain the final dissipative estimate
		\begin{equation} \label{final dissipative bound}
			\begin{aligned}
				\norm{\psi}_{L^2_{[0,T]} \dot{H}^{2s+1}_x}^2 &\lesssim W_0 + \left( W_0+Z_0+S_0^{p+1} \right) \norm{u}_{L^2_{[0,T]} H^2_x}^2 + \left( Z_0 + S_0^{\frac{p}{2}+2} \right) \norm{u}_{L^2_{[0,T]} H^2_x}^2 \\
				&\quad + Z_0 \left( 1-e^{-\frac{T}{C}} \right) + S_0^{\frac{p}{2}+1}\left(1- \left(1+S_0^{\frac{p}{2}} T\right)^{-1-\frac{2}{p}} \right) \\
				&\lesssim W_0+Z_0+S_0^{\frac{p}{2}+1}.
			\end{aligned}
		\end{equation} 
		This shows that with small enough data one can achieve an arbitrarily small value (independent of $T$) for this highest-order dissipation.
		
		Similarly to ~\eqref{improved Y dissipative estimates}, it is possible to also get a time-decaying estimate by integrating ~\eqref{highest order energy LHS inhomogenous} over the time interval $[t,2t]$ for $t\ge 1$. This leads to
		\begin{equation} \label{improved highest dissipative estimate}
			\begin{aligned}
				\norm{\psi}_{L^2_{[t,2t]} \dot{H}^{2s+1}_x}^2 &\lesssim \norm{\psi(t)}_{\dot{H}^{2s}_x}^2 + \left( \norm{\psi(t)}_{\dot{H}^{2s}_x}^2 + Z_0 e^{-\frac{t}{C}} + \frac{S_0^{\frac{p}{2}+2}}{\left(1+S_0^{\frac{p}{2}}t\right)^{1+\frac{4}{p}}} \right) \norm{u}_{L^2_{[t,2t]} \dot{H}^2_x}^2 \\ 
				&\quad + \int_t^{2t} Z_0 e^{-\frac{s}{C}} + \frac{S_0^{p+1}}{\left(1+S_0^{\frac{p}{2}} s\right)^{2+\frac{2}{p}}} ds \\
				&\lesssim \left(W_0 + Z_0 \right) e^{-\frac{t}{2C}} +  \frac{S_0^{\frac{p}{2}+1}}{\left(1+S_0^{\frac{p}{2}}t\right)^{1+\frac{2}{p}}},
			\end{aligned}
		\end{equation}
		where we have used ~\eqref{improved Y dissipative estimates} and ~\eqref{homogeneous high norm bound 2}, and retained only the slowest decaying terms.
	\end{proof}
	
	The high-norm control in ~\eqref{final dissipative bound} and ~\eqref{improved highest dissipative estimate} is important because the inequalities can be translated into the desired bounds (on two fewer derivatives) for $B\psi$. Indeed,
	\begin{equation} \label{splitting Bpsi into low and high norms}
		\begin{aligned}
			\norm{B\psi}_{H^{2s-1}_x}^2 &\lesssim \norm{B\psi}_{L^2_x}^2 + \norm{B\psi}_{\dot{H}^{2s-1}_x}^2 \\
			&\lesssim \norm{B\psi}_{L^2_x}^2 + \norm{\psi}_{\dot{H}^{2s+1}_x}^2 + \norm{\abs{u}^2\psi}_{H^{2s-1}_x}^2 + \norm{u\cdot\nabla\psi}_{H^{2s-1}_x}^2 + \norm{\abs{\psi}^p\psi}_{H^{2s-1}_x}^2 ,
		\end{aligned}
	\end{equation}
	where for the last three terms, we replaced the homogeneous Sobolev norms by the larger inhomogeneous norms. 
	Combining the analysis in ~\eqref{u^2 psi in highest norm}--\eqref{psi^p psi in highest norm} with ~\eqref{superfluid mass bound}, ~\eqref{Y dissipative estimates}, ~\eqref{final dissipative bound}, and ~\eqref{splitting Bpsi into low and high norms}, we get the sought-after \emph{dissipation bound}
	\begin{equation} \label{Bpsi sought-after bound}
		\norm{B\psi}_{L^2_{[0,T]} H^{2s-1}_x}^2 \lesssim W_0 + Z_0 +  S_0^{\frac{p}{2}+1} .
	\end{equation}
	Since $2s-1 = \frac{3}{2}$, Sobolev embedding allows us to conclude that $B\psi$ is bounded in $L^2_{[0,T]}L^{\infty}_x$. Similarly, integrating ~\eqref{splitting Bpsi into low and high norms} over $[t,2t]$ for $t\ge 1$, we get
	\begin{equation} \label{improved Bpsi sought-after bound}
		\norm{B\psi}_{L^2_{[t,2t]} H^{2s-1}_x}^2 \lesssim \left(W_0 + Z_0 \right) e^{-\frac{t}{2C}} + \frac{S_0^{\frac{p}{2}+1}}{\left(1+S_0^{\frac{p}{2}}t\right)^{\frac{2}{p}}}.
	\end{equation}
	The estimates in ~\eqref{Bpsi sought-after bound} and ~\eqref{improved Bpsi sought-after bound} are used to ensure that the density remains bounded from below.
	
	\subsection{Ensuring positive density} \label{ensuring positive density}
	We now have all the a priori estimates to return to ~\eqref{constraint to choose existence time}. For it to hold true, a sufficient condition is
	\begin{equation} \label{constraint to choose existence time 3}
		\left(\norm{\psi}_{L^2_{[0,T]} L^2_x} + \norm{\Delta \psi}_{L^2_{[0,T]} L^2_x} \right) \norm{B\psi}_{L^2_{[0,T]} H^{\frac{3}{2}}_x}  \lesssim \mi-\mf.
	\end{equation}
	Depending on the value of $p$, we now divide the analysis into several cases: $1\le p<2$, $p=2$, $2<p<4$, $p=4$, and $p>4$. 
	
	\subsubsection{The case $1\le p<2$} \label{case 1<p<2}
	Owing to the Poincar\'e inequality and ~\eqref{Y dissipative estimates}, we have
	\begin{equation} \label{Delta psi in L^2_tL^2_x}
		\norm{\Delta \psi}_{L^2_{[0,T]} L^2_x}^2 \lesssim \norm{D^3 \psi}_{L^2_{[0,T]} L^2_x}^2 \lesssim Z_0 + S_0^{p+1} ,
	\end{equation}
	and this bound holds for all $p\ge 1$. For the first term of ~\eqref{constraint to choose existence time 3}, we integrate ~\eqref{superfluid mass bound}, yielding
	\begin{equation} \label{psi in L^2_t L^2_x}
		\norm{\psi}_{L^2_{[0,T]} L^2_x}^2 \lesssim \frac{S_0^{1-\frac{p}{2}}}{2-p} \left(1- \frac{1}{\left(1+S_0^{\frac{p}{2}} T\right)^{\frac{2}{p}-1}}\right) \lesssim S_0^{1-\frac{p}{2}} ,
	\end{equation}
	since $\frac{2}{p} > 1$. From ~\eqref{Bpsi sought-after bound}, ~\eqref{Delta psi in L^2_tL^2_x}, and ~\eqref{psi in L^2_t L^2_x}, we conclude that the condition in ~\eqref{constraint to choose existence time 3} can be achieved if $W_0 + Z_0 + S_0^{1-\frac{p}{2}}$ is sufficiently small. Thus, the density satisfies $\mf\le \rho \le \Mi+\mi-\mf$ for all $T>0$, as long as the initial data are small enough. 
	
	For $p\ge 2$, the integral of the superfluid mass, i.e., $\norm{\psi(t)}_{L^2_x}^2$, cannot be bounded uniformly in $[0,T]$. This is where the decaying estimates in ~\eqref{improved Y dissipative estimates} and ~\eqref{improved Bpsi sought-after bound} prove to be useful.

	\subsubsection{The case $p=2$} \label{case p=2}
	We split the time integral in ~\eqref{constraint to choose existence time 3} over the ranges $0\le t\le 1$ (short-time) and $t\ge 1$ (long-time). We start with the long-time estimate the LHS of ~\eqref{constraint to choose existence time 3} with $p=2$. For the first term, we have
	\begin{equation} \label{long-time psi in L^2 L^2, p=2}
		\int_t^{2t} \norm{\psi}_{L^2_x}^2 \lesssim \int_t^{2t} \frac{S_0}{1+S_0 t} \lesssim \log{\left(\frac{1+2S_0 t}{1+S_0 t}\right)} \lesssim \log{2} .
	\end{equation}
	Using the Poincar\'e inequality and ~\eqref{improved Y dissipative estimates} gives
	\begin{equation} \label{long-time Delta psi in L^2 L^2, p=2}
		\int_t^{2t} \norm{\Delta\psi}_{L^2_x}^2 \lesssim \int_t^{2t} \norm{D^3\psi}_{L^2_x}^2 \lesssim Z_0 e^{-\frac{t}{C}} + \frac{S_0^2}{1+S_0 t}.
	\end{equation}
	From ~\eqref{improved Bpsi sought-after bound}, ~\eqref{long-time psi in L^2 L^2, p=2}, and ~\eqref{long-time Delta psi in L^2 L^2, p=2}, we obtain
	\begin{equation*}
		\begin{aligned}
			I(t) := \int_t^{2t} \norm{\psi}_{L^{\infty}_x} \norm{B\psi}_{L^{\infty}_x} &\lesssim \left(\norm{\psi}_{L^2_{[t,2t]} L^2_x} + \norm{\Delta \psi}_{L^2_{[t,2t]} L^2_x} \right) \norm{B\psi}_{L^2_{[t,2t]} H^{\frac{3}{2}}_x} \\ 
			&\lesssim \left( \log{2} + Z_0 e^{-\frac{t}{C}} +  \frac{S_0^2}{1+S_0 t} \right)^{\frac{1}{2}} \left( \left(W_0 + Z_0 \right) e^{-\frac{t}{2C}} + \frac{S_0^2}{1+S_0 t} \right)^{\frac{1}{2}} \\
			&\lesssim \left(W_0 + Z_0 \right)^{\frac{1}{2}} e^{-\frac{t}{4C}} + \frac{S_0}{(1+S_0 t)^{\frac{1}{2}}} \\
			&\lesssim \frac{(W_0+Z_0)^{\frac{1}{2}}+S_0^{\frac{1}{2}}}{t^{\frac{1}{2}}}.
		\end{aligned}
	\end{equation*}
	This leads us to
	\begin{equation} \label{long-time control of constraint p=2}
		\int_1^{2^{N+1}} \norm{\psi}_{L^{\infty}_x} \norm{B\psi}_{L^{\infty}_x} = \sum_{k=0}^N I(2^k) \lesssim (W_0+Z_0+S_0)^{\frac{1}{2}} \sum_{k=0}^N \frac{1}{(2^k)^{\frac{1}{2}}} \lesssim (W_0+Z_0+S_0)^{\frac{1}{2}} ,
	\end{equation}
	which is the long-time contribution (independent of $N$) of the constraint in ~\eqref{constraint to choose existence time 3}. It can be made as small as required with an appropriate choice of $W_0+Z_0+S_0$.
	
	Finally, we verify the short-time control as well. The superfluid mass bound in ~\eqref{superfluid mass bound} means that
	\begin{equation} \label{short-time psi in L^2 L^2, p=2}
		\int_0^1 \norm{\psi}_{L^2_x}^2 \lesssim \int_0^1 \frac{S_0}{1+S_0 t} \lesssim \log{(1+S_0)}.
	\end{equation}
	Similarly, using ~\eqref{Y dissipative estimates}, we get
	\begin{equation} \label{short-time Delta psi in L^2 L^2, p=2}
		\int_0^1 \norm{\Delta\psi}_{L^2_x}^2 \lesssim \int_0^1 \norm{D^3\psi}_{L^2_x}^2 \lesssim Z_0 + S_0^3.
	\end{equation}
	From ~\eqref{Bpsi sought-after bound}, ~\eqref{short-time psi in L^2 L^2, p=2}, and ~\eqref{short-time Delta psi in L^2 L^2, p=2}, we have
	\begin{equation} \label{short-time LHS of constraint, p=2}
		\int_0^1 \norm{\psi}_{L^{\infty}_x} \norm{B\psi}_{L^{\infty}_x} \lesssim \left( \log{(1+S_0)} + Z_0 + S_0^3 \right)^{\frac{1}{2}} (W_0 + Z_0 + S_0^2)^{\frac{1}{2}} ,
	\end{equation}
	which can be made small enough to satisfy ~\eqref{constraint to choose existence time 3}. This lets us conclude that the density is bounded from below uniformly in time, for the case $p=2$. Thus, we have the necessary global bound.

	\subsubsection{The case $2<p<4$} \label{case 2<p<4}
	We begin, once again, with the long-time analysis, i.e, for $t\ge 1$. From ~\eqref{superfluid mass bound}, we have
	\begin{equation} \label{improved psi in L^2 L^2}
		\norm{\psi}_{L^2_{[t,2t]} L^2_x}^2 \lesssim \frac{S_0^{1-\frac{p}{2}}}{\left(1+S_0^{\frac{p}{2}} t\right)^{\frac{2}{p}-1}}.
	\end{equation}
	Using the Poincar\'e inequality and ~\eqref{improved Y dissipative estimates},
	\begin{equation} \label{improved Delta psi in L^2 L^2}
		\norm{\Delta\psi}_{L^2_{[t,2t]} L^2_x}^2 \lesssim \norm{D^3\psi}_{L^2_{[t,2t]} L^2_x}^2 \lesssim Z_0 e^{-\frac{t}{C}} + \frac{S_0^{\frac{p}{2}+1}}{\left(1+S_0^{\frac{p}{2}}t\right)^{\frac{2}{p}}}.
	\end{equation}
	Combining ~\eqref{improved Bpsi sought-after bound}, ~\eqref{improved psi in L^2 L^2} and ~\eqref{improved Delta psi in L^2 L^2}, we have
	\begin{equation}
		\begin{aligned}
			I(t) &\lesssim \left( \frac{S_0^{\frac{1}{2}-\frac{p}{4}}}{\left(1+S_0^{\frac{p}{2}} t\right)^{\frac{1}{p}-\frac{1}{2}}} + Z_0^{\frac{1}{2}} e^{-\frac{t}{2C}} +  \frac{S_0^{\frac{p}{4}+\frac{1}{2}}}{\left(1+S_0^{\frac{p}{2}}t\right)^{\frac{1}{p}}} \right) \left( \left(W_0 + Z_0 \right)^{\frac{1}{2}} e^{-\frac{t}{4C}} + \frac{S_0^{\frac{p}{4}+\frac{1}{2}}}{\left(1+S_0^{\frac{p}{2}}t\right)^{\frac{1}{p}}} \right) \\
			&\lesssim \left( \frac{1}{ t^{\frac{1}{p}-\frac{1}{2}}} + Z_0^{\frac{1}{2}} e^{-\frac{t}{2C}} +  \frac{S_0^{\frac{p}{4}}}{t^{\frac{1}{p}}} \right) \left( \left(W_0 + Z_0 \right)^{\frac{1}{2}} e^{-\frac{t}{4C}} + \frac{S_0^{\frac{p}{4}}}{t^{\frac{1}{p}}} \right) \\
			&\lesssim \frac{(W_0+Z_0)^{\frac{1}{2}}+S_0^{\frac{p}{4}}}{t^{\frac{2}{p}-\frac{1}{2}}}.
		\end{aligned}
	\end{equation}
	Once again, the slowest decaying term is the dominant one. 
	Therefore, we have
	\begin{equation} \label{long-time control of constraint}
		\int_1^{2^{N+1}} \norm{\psi}_{L^{\infty}_x} \norm{B\psi}_{L^{\infty}_x} = \sum_{k=0}^N I(2^k) \lesssim \left((W_0+Z_0)^{\frac{1}{2}}+S_0^{\frac{p}{4}}\right) \sum_{k=0}^N \frac{1}{(2^k)^{\frac{2}{p}-\frac{1}{2}}} \lesssim \left(W_0+Z_0+S_0^{\frac{p}{2}}\right)^{\frac{1}{2}}.
	\end{equation}
	The sum converges (uniformly in $N$) because $p<4$. Hence, we obtain good long-time control of the LHS of ~\eqref{constraint to choose existence time 3} for $2<p<4$. 
	
	What remains is to check that we also maintain short-time control. To this end, we have from ~\eqref{superfluid mass bound},
	\begin{equation} \label{short-time psi in L^2 L^2, 2<p<4}
		\int_0^1 \norm{\psi}_{L^2_x}^2 \lesssim \int_0^1 \frac{S_0}{\left(1+S_0^{\frac{p}{2}} t\right)^{\frac{2}{p}}} \lesssim S_0^{1-\frac{p}{2}} \left( (1+S_0^{\frac{p}{2}})^{1-\frac{2}{p}} -1 \right) ,
	\end{equation}
	and from ~\eqref{Y dissipative estimates},
	\begin{equation} \label{short-time Delta psi in L^2 L^2, 2<p<4}
		\int_0^1 \norm{\Delta\psi}_{L^2_x}^2 \lesssim \int_0^1 \norm{\Delta\psi}_{L^2_x}^2 \lesssim Z_0 + S_0^{p+1} .
	\end{equation}
	Combining ~\eqref{Bpsi sought-after bound} with ~\eqref{short-time psi in L^2 L^2, 2<p<4} and ~\eqref{short-time Delta psi in L^2 L^2, 2<p<4} yields 
	\begin{equation*}
		\begin{aligned}
			\int_0^1 \norm{\psi}_{L^{\infty}_x} \norm{B\psi}_{L^{\infty}_x} &\lesssim \left( S_0^{\frac{1}{2}-\frac{p}{4}} \left( \left(1+S_0^{\frac{p}{2}}\right)^{1-\frac{2}{p}} - 1 \right)^{\frac{1}{2}} + Z_0^{\frac{1}{2}} + S_0^{\frac{p+1}{2}} \right) \left( (W_0 + Z_0)^{\frac{1}{2}} + S_0^{\frac{p}{4}+\frac{1}{2}} \right) \\
			&\le C\left( S_0^{\frac{1}{2}-\frac{p}{4}} \left( 1 + CS_0^{\frac{p}{2}-1} - 1 \right)^{\frac{1}{2}} + Z_0^{\frac{1}{2}} + S_0^{\frac{p+1}{2}} \right) \left( (W_0 + Z_0)^{\frac{1}{2}} + S_0^{\frac{p}{4}+\frac{1}{2}} \right) \\
			&\le C\left( 1 + Z_0^{\frac{1}{2}} + S_0^{\frac{p+1}{2}} \right) \left( (W_0 + Z_0)^{\frac{1}{2}} + S_0^{\frac{p}{4}+\frac{1}{2}} \right) \\
			&\lesssim \left(W_0 + Z_0 + S_0^{\frac{p}{2}+1} \right)^{\frac{1}{2}},
		\end{aligned}
	\end{equation*}
	which is the short-time control we are seeking. This implies global solutions, since the density is bounded from below uniformly in time.

	\subsubsection{The case $p\ge 4$} \label{case p>=4}
	The arguments for short-time control in Section ~\ref{case 2<p<4} remain valid even for $p\ge 4$. However, the long-time estimates breaks down. Specifically, the geometric series in ~\eqref{long-time control of constraint} diverges. We see that for $T=2^{N+1}$, 
	\begin{equation} \label{long-time loss of control}
		\int_1^T \norm{\psi}_{L^{\infty}_x} \norm{B\psi}_{L^{\infty}_x} \lesssim \sum_{k=0}^N \frac{W_0^{\frac{1}{2}}+Z_0^{\frac{1}{2}}+S_0^{\frac{p}{4}}}{(2^k)^{\frac{2}{p}-\frac{1}{2}}} \lesssim \left\{ \begin{aligned}
			& \left(W_0+Z_0+S_0^2 \right)^{\frac{1}{2}} \log{T} &, \quad p=4 \\
			& \left(W_0+Z_0+S_0^{\frac{p}{2}} \right)^{\frac{1}{2}} T^{\frac{p-4}{2p}} &, \quad p>4.
		\end{aligned} \right.
	\end{equation}
	Therefore, in this scenario, global-in-time estimates elude us due to the logarithmic/polynomial dependence on $T$. We can, however, guarantee almost global existence of solutions. Given a set of system parameters, we can ensure that $\rho\ge \mf$ for any finite time $T>0$ as long as we start from small enough initial data (depending on $T$). In other words, if the size of the data is $\veps$, then we have $T\sim e^{\frac{1}{\sqrt{\veps}}}$ for $p=4$ and $T\sim \veps^{-\frac{p}{p-4}}$ for $p>4$. This is the scaling expressed in Theorem ~\ref{almost global existence}.

	\section{Existence of weak solutions (Proof of Theorems ~\ref{global existence} and ~\ref{almost global existence})} \label{existence proof}
	Having derived the required a priori estimates, we now establish the existence of a weak solution for a truncated form of the governing equations, and then pass to the limit.
	
	\subsection{Constructing the semi-Galerkin scheme}
	The finite-dimensional wavefunction and velocity are constructed using eigenfunctions of the Laplacian and the Leray-projected Laplacian, respectively. 
	
	\subsubsection{The approximate wavefunction} \label{Constructing the wavefunction and the Dirichlet penta-Laplacian}
	Consider the negative Laplacian $-\Delta$ on the torus $\T^2$, with the domain $D(-\Delta) = H^2$. It has a discrete set of non-negative and non-decreasing eigenvalues $\{\beta_j\}$, and the corresponding eigenfunctions $\{b_j\}\in C^{\infty}(\T^2)$ can be chosen to be orthonormal in $L^2_x$ and orthogonal in $H^1_x$. We define the approximate wavefunction as
	\begin{equation} \label{truncated wavefunction definition}
		\psi^N (t,x) = \sum_{k=0}^N d^N_k(t) b_k(x)
		,
	\end{equation}
	for $N\in\N\cup\{0\}$ and $d^N_k(t)\in\C$.
	
	\subsubsection{The approximate velocity}
	We consider the Leray-projected Laplacian (or Stokes operator) $A = -\Leray\Delta$ with the domain $D(A) = L^2_{\text{d}} \cap H^2$ (see \cite[Chapter 2]{Robinson2016TheEquations}, for instance).
	
	The Stokes operator (like the Laplacian) has a discrete set of non-negative and non-decreasing eigenvalues $\{\alpha_j\}$, and the corresponding divergence-free, vector-valued eigenfunctions $\{a_j\}\in C^{\infty}(\T^2)$ can be chosen to be orthonormal in $L^2_{\text{d},x}$ and orthogonal in $H^1_x$. We define the approximate velocity as
	\begin{equation} \label{truncated velocity definition}
		u^N (t,x) = \sum_{k=0}^N c^N_k(t) a_k(x)
		,
	\end{equation}
	for $N\in\N\cup\{0\}$ and $c^N_k(t)\in\R$.

	\subsection{The initial conditions} \label{the initial conditions}
	
	\subsubsection{The initial wavefunction and initial velocity}
	We begin by defining $P^N$ (respectively, $Q^N$) to be the projections onto the space spanned by the first $N+1$ eigenfunctions of $A$ (respectively, $-\Delta$). Then, we truncate the initial conditions for the velocity and wavefunction accordingly:
	\begin{equation} \label{truncated initial conditions velocity and wavefunction}
		u_0^N := P^N u_0 \quad , \quad \psi_0^N := Q^N \psi_0 .
	\end{equation}
	Since $u_0 \in H^{1}_{\text{d}}(\T^2)$ and $\psi_0 \in H^{\frac{5}{2}}(\T^2)$, it is necessary to establish that the truncated initial conditions converge to the actual ones in the relevant norms.
	\begin{lem} [The projections $Q^N$ and $P^N$ are convergent] \label{Q^N and P^N are convergent}
		If $\psi \in H^r_x$ and $u \in H^s_{\text{d},x}$ for any $0<r,s<\infty$, then 
		\begin{enumerate}
			\item $\norm{Q^N \psi}_{H^r_x} \lesssim \norm{\psi}_{H^r_x}$ and $Q^N \psi \xrightarrow[N\rightarrow\infty]{H^r}\psi$, and
			
			\item $\norm{P^N u}_{H^s_{x}} \lesssim \norm{u}_{H^s_{x}}$ and $P^N u \xrightarrow[N\rightarrow\infty]{H^s}u$ .
		\end{enumerate}
	\end{lem}
	\noindent The proof utilizes the equivalence of norms between Sobolev spaces and fractional powers of the negative Laplacian/Stokes operator (see Theorem 2.27 in \cite{Robinson2016TheEquations}). Given the regularity of $\psi_0$ and $u_0$, we deduce the convergence of the approximate initial conditions by applying Lemma ~\ref{Q^N and P^N are convergent}.

	\subsubsection{The initial density} \label{the initial density}
	Given the initial density field $\rho_0 \in L^{\infty}_x \subset L^r_x$ for $1\le r<\infty$, we consider an approximating sequence $\rho^N_0 \in C^1_x$, such that $\rho_0^N \xrightarrow[N\rightarrow\infty]{L^r}\rho$, and $\mi\le \rho_0^N \le \Mi$. (Recall that $\mi\le \rho_0 \le \Mi$.) This approximating sequence may be constructed by mollification.
	
	

	\subsection{Approximate equations}
	
	\subsubsection{The continuity equation}
	
	Having described the (approximate) initial conditions and the semi-Galerkin scheme, we now establish the existence of solutions to the ``approximate" equations, starting with the continuity equation. It is given by
	\begin{equation} \label{approximate continuity equation}
		\begin{aligned}
			\partial_t \rho^N + u^N\cdot\nabla\rho^N &= 2 \lambda\Re (\overline{\psi^N}B^N\psi^N) , \\
			\rho^N(0,x) &= \rho^N_0(x) ,
		\end{aligned}
	\end{equation}
	where $B^N = -\frac{1}{2}\Delta + \frac{1}{2}\abs{u^N}^2 + iu^N\cdot\nabla + \mu\abs{\psi^N}^p$. Just as in ~\eqref{constraint to choose existence time}, we see that the constraint that fixes the local existence time $T_N$ for ~\eqref{approximate continuity equation} is
	\begin{equation} \label{constraint to choose approximate existence time}
		2\lambda\norm{\psi^N}_{L^{\infty}_{[0,T_N]}L^{\infty}_x} \norm{B^N\psi^N}_{L^2_{[0,T_N]}L^{\infty}_x} \le \mi-\mf .
	\end{equation}
	Since the norms in ~\eqref{constraint to choose approximate existence time} are bounded by the size of the initial data, the time $T_N$ is independent of ~$N$. Hence, we use $T$ to denote the time of existence, with $T$ arbitrarily large for $1\le p<4$ and~$T$ bounded for $p\ge 4$ (as specified in Theorem ~\ref{almost global existence}).
	
	We now establish the analogs of Lemmas 2.2 and 2.3 from \cite{Kim1987WeakDensity}. These constitute the existence of a unique solution to \eqref{approximate continuity equation} and a Picard iteration scheme for the same, respectively.
	\begin{lem} \label{existence of solutions to approx continuity equation}
		Let $u^N \in C^0_{[0,T]} C^1_x$ and $\overline{\psi^N}B^N \psi^N\in L^1_{[0,T]} L^{\infty}_x$ (uniformly in $N$), with $\nabla\cdot u^N (t,\T^2)=0$ for $t\in [0,T]$. Then, \eqref{approximate continuity equation} has a unique solution $\rho^N \in C^1_{[0,T]}C^1_x$.
	\end{lem}
	\begin{proof}
		Consider the evolution equation for the characteristics of the flow,
		\begin{equation} \label{evolution of the characteristics}
			\begin{aligned}
				\frac{dx^N}{dt} &= u^N(t,x^N(t)) , \\
				x^N(0) &= y^N \in \T^2 .
			\end{aligned}
		\end{equation}
		Since $u^N \in C^0_{[0,T]} C^1_x$, there exists a unique solution $x^N(t,y^N) \in C^1_{[0,T]} C^1_{x}$.
		Owing to the incompressibility of the flow $u^N$, it follows that $\det\left( \frac{\partial x^N_i}{\partial y^N_j}\right) = 1$, allowing us to conclude that the characteristics are $C^1$ diffeomorphisms and therefore, invertible. This means
		\begin{equation*}
			y^N = y^N(t,x^N) := S^{-1}_t x^N
		\end{equation*}
		is well-defined. We now write the solution to \eqref{approximate continuity equation} along characteristics as    
		\begin{equation} \label{solution to approx continuity eqn, along characteristics}
			\rho^N(t,x) = \rho^N_0\left(y^N(t,x)\right) + 2 \lambda\int_0^t \Re\left( \overline{\psi^N}B^N \psi^N \right) \left( \tau,y^N(t-\tau,x) \right) d\tau .
		\end{equation}
		That \eqref{solution to approx continuity eqn, along characteristics} uniquely solves \eqref{approximate continuity equation} can be verified using the property of the ``inverse-characteristics" $y(t,x)$. For any $\tau\in\R$,
		\begin{equation} \label{inverse characteristics proof}
			\begin{aligned}
				\frac{\partial}{\partial t}y(t-\tau,x) &= \lim_{\Delta t\rightarrow 0} \frac{y(t-\tau+\Delta t,x) - y(t-\tau,x)}{\Delta t} \\
				&= \lim_{\Delta t\rightarrow 0} \frac{x(t+\Delta t,y) - x(t,y)}{\Delta t} \cdot \frac{y(t-\tau+\Delta t,x) - y(t-\tau,x)}{x(t+\Delta t,y) - x(t,y)} \\
				&= u(t,x) \cdot \frac{\partial_t y(t-\tau,x)}{\partial_t x(t,y)} = -u(t,x)\cdot\nabla_{x}y(t-\tau,x) ,
			\end{aligned}
		\end{equation}
		where the last equality is due to Euler's chain rule.
	\end{proof}
	\noindent Now, we consider a convergent sequence of velocities and wavefunctions that belong to the finite-dimensional subspaces spanned by the truncated Galerkin scheme. Given such a convergent sequence, we show that the sequence of density fields satisfying \eqref{approximate continuity equation} is also convergent, and this shall be used to complete a contraction mapping argument below.
	\begin{lem} \label{convergence of solutions to approx continuity equation}
		For $n\in\N$, let $u^N_n \in C^0_{[0,T]} C^1_{x}$ and $\overline{\psi^N_n}B^N_n \psi^N_n\in L^1_{[0,T]} L^{\infty}_x$ (uniformly in $n$), with $\nabla\cdot u^N_n (t,\T^2) = 0$ for $t\in [0,T]$. Denote by $\rho^N_n \in C^1_{[0,T]}C^1_x$ the unique solution to the system
		\begin{equation} \label{sequence approximate continuity equation}
			\begin{aligned}
				\partial_t \rho^N_n + u^N_n\cdot\nabla\rho^N_n &= 2 \lambda\Re (\overline{\psi^N_n}B^N_n\psi^N_n) , \\
				\rho^N_n(0,x) &= \rho^N_0(x) \in C^1_x .
			\end{aligned}
		\end{equation}
		If $u^N_n \xrightarrow[n\rightarrow\infty]{C^0_{[0,T]}C^1_{x}} u^N$ and $\psi^N_n \xrightarrow[n\rightarrow\infty]{C^0_{[0,T]}C^3_{x}} \psi^N$, then $\rho^N_n \xrightarrow[n\rightarrow\infty]{C^0_{[0,T]}C^0_{x}} \rho^N$, where $\rho^N$ solves \eqref{approximate continuity equation}.
	\end{lem}
	\begin{proof}
		We begin by defining $\Psi^N_n := 2 \lambda\Re(\overline{\psi^N_n} B^N_n\psi^N_n)$. Since $u^N_n\in C^0_t C^1_{x}$, there exists a sequence of characteristics $x^N_n(t,y)\in C^1_t C^1_{x}$ corresponding to the flow, i.e., solving $\frac{dx^N_n}{dt} = u^N_n (t,x^N_n)$ with $x^N_n(0,y) = y$. The assumed convergence of $u^N_n$ allows us to conclude that $x^N_n \xrightarrow[n\rightarrow\infty]{C^1_t C^1_{x}}x^N$.
		Consider the map $y\mapsto x^N_n(t,y)$ and define its inverse $y^N_n(t,x)$; this is just the inverse of the characteristic, i.e., if the flow were reversed. Due to the flow being incompressible, we know that the matrix $\frac{\partial y^N_n}{\partial x}$ is invertible. Also, as shown in the proof of the previous lemma, $\frac{\partial}{\partial t}y^N_n = -u^N_n\cdot\nabla_{x}y^N_n$. This implies that the derivatives of $y^N_n$ with respect to both space and time are bounded uniformly in $n$, $t$ and ~$x$. Thus, by the Arzela-Ascoli theorem, we can extract a subsequence that converges uniformly: $y^N_n\xrightarrow[n\rightarrow\infty]{C^0_t C^0_{x}}y^N$. Just as in ~\eqref{inverse characteristics proof}, we can show that the solution to \eqref{sequence approximate continuity equation} is
		\begin{equation} \label{solution to sequence approx continuity eqn, along characteristics}
			\rho^N_n(t,x) = \rho^N_0\left(y^N_n(t,x)\right) + \int_0^t \Psi^N_n \left( \tau,y^N_n(t-\tau,x) \right) d\tau .
		\end{equation}
		Therefore,
		\begin{equation*}
			\begin{aligned}
				\rho^N_n(t,x) - \rho^N(t,x) &= \rho^N_0\left(y^N_n(t,x)\right) + \int_0^t \Psi^N_n \left( \tau,y^N_n(t-\tau,x) \right) d\tau \\ 
				&\quad - \rho^N_0\left(y^N(t,x)\right) - \int_0^t \Psi^N \left( \tau,y^N(t-\tau,x) \right) d\tau
				,
			\end{aligned}
		\end{equation*}
		which leads to
		\begin{equation*}
			\begin{aligned}
				\abs{\rho^N_n - \rho^N}_{C^0_{t,x}} &\le \abs{\rho^N_0\left(y^N_n\right) - \rho^N_0\left(y^N\right)}_{C^0_{t,x}} + T \abs{\Psi^N_n \left( t,y^N_n \right) - \Psi^N \left( t,y^N \right)}_{C^0_{t,x}} \\
				&\le \norm{\nabla\rho^N_0}_{L^{\infty}_x}\abs{y^N_n - y^N}_{C^0_{t,x}} \\ 
				&\quad + T\left( \norm{\nabla\Psi^N_n}_{L^{\infty}_t L^{\infty}_x}\abs{y^N_n - y^N}_{C^0_{t,x}} + \abs{\Psi^N_n - \Psi^N}_{C^0_{t,x}} \right) \\
				&\xrightarrow[n\rightarrow\infty]{} 0.
			\end{aligned}
		\end{equation*}
		Given the convergence of $y^N_n$ derived above, and because $\rho^N_0\in C^1_x$, the first term on the RHS vanishes. The second and third terms vanish on account of the following argument. Note that $\Psi^N_n$ has its highest order term of the form $\psi^N_n\Delta\psi^N_n$ (second derivative), and so the assumed convergence of ~$\psi^N_n$ in the $C^0_t C^3_x$ norm implies that $\Psi^N_n$ converges in $C^0_t C^1_x$. This also guarantees that ~$\norm{\nabla\Psi^N_n}_{L^{\infty}_{t} L^{\infty}_x}$ is finite, uniformly in $n$. 
	\end{proof}

	\subsubsection{The Navier-Stokes equation}
	We now consider an ``approximate momentum equation", composed of the approximate wavefunction and velocity fields defined by \eqref{truncated wavefunction definition} and \eqref{truncated velocity definition}, respectively. Namely,
	\begin{equation} \label{approximate NSE}
		P^N\left(\rho^N\partial_t u^N + \rho^N u^N\cdot\nabla u^N - \nu\Delta u^N \right) = - 2\lambda P^N\left(\Im\left( \nabla\overline{\psi^N}B^N\psi^N \right) + u^N\Re\left(\overline{\psi^N}B^N\psi^N \right) \right) .
	\end{equation}
	Recall that the incompressiblity condition is built-in, because the eigenfunction basis used to construct the velocity fields are divergence-free. Now, taking the $L^2$ inner product of \eqref{approximate NSE} with $a_j(x)$ for $0\le j\le N$, we arrive at a system of equations for the coefficients describing the time-dependence of the approximate velocity fields, as
	\begin{equation} \label{expansion of approximate NSE}
		\sum_{k=0}^N R^N_{jk}(t) \frac{d}{dt}c^N_k(t) = -\nu \alpha_j c^N_j(t) - \sum_{k,l=0}^N \mathcal{N}^N_{jkl}(t) c^N_k(t) c^N_l(t) - 2 \lambda S^N_j(t,c^N) ,
	\end{equation}
	where
	\begin{equation*}
		R^N_{jk}(t) = \int_{\T^2}\rho^N a_j\cdot a_k , \quad \mathcal{N}^N_{jkl}(t) = \int_{\T^2}\rho^N \left( a_k\cdot\nabla \right) a_l \cdot a_j ,
	\end{equation*}
	and
	\begin{equation*}
		S^N_j(t,c^N) = \int_{\T^2}a_j\cdot \left( \Im\left( \nabla\overline{\psi^N}B^N\psi^N \right) + u^N\Re\left(\overline{\psi^N}B^N\psi^N \right) \right) .
	\end{equation*} 
	Since we have both lower and upper bounds on the density in the chosen interval of time, we can use Lemma 2.5 in \cite{Kim1987WeakDensity} to show that the matrix $R^N(t)$ is invertible. Therefore, we arrive at
	\begin{equation} \label{evolution equation for c^N}
		\frac{d}{dt}c^N = -\nu (R^N)^{-1} D\cdot c^N - (R^N)^{-1}\left( \mathcal{N}^N : c^N \otimes c^N \right) -2 \lambda (R^N)^{-1} S^N(t,c^N) ,
	\end{equation} 
	which is the desired evolution equation (written vectorially) for the coefficients $c^N_j(t)$.
	
	\subsubsection{The nonlinear Schr\"odinger equation}
	
	As in the previous section, we derive an evolution equation for the coefficients of the approximate wavefunction, by considering an ``approximate NLS". Namely,
	\begin{equation} \label{approximate NLS}
		\partial_t \psi^N = -\frac{1}{2i}\Delta \psi^N - Q^N\left(  \lambda B^N_L \psi^N + (\lambda + i)\mu \abs{\psi^N}^p \psi^N \right) .
	\end{equation} 
	Recall that $B_L = B - \mu \abs{\psi}^p$, i.e., the linear (in $\psi$) part of the coupling operator. Performing an $L^2$ inner product with $b_j(x)$, we get
	\begin{equation}
		\frac{d}{dt}d^N_j (t) = \frac{1}{2i} \beta_j d^N_j(t) -  \lambda\sum_{k=0}^N L^N_{jk}(t) d^N_k(t) - ( \lambda + i)\mu \sum_{k,l,m=0}^N G_{jklm} \left(d^N_k \overline{d^N_l}\right)^{\frac{p}{2}} d^N_m (t)
		,
	\end{equation} 
	where
	\begin{equation*}
		L^N_{jk}(t) = \int_{\T^2} b_j B^N_L b_k = \frac{1}{2}\beta_j\delta_{jk} + \frac{1}{2}\int_{\T^2} \abs{u^N}^2 b_j b_k + i\int_{\T^2} \left(u^N\cdot\nabla b_k\right) b_j 
	\end{equation*}
	and
	\begin{equation*}
		G_{jklm} = \int_{\T^2} b_j (b_k b_l)^{\frac{p}{2}} b_m .
	\end{equation*}
	Written vectorially, the evolution equation for the coefficients $d^N_j(t)$ becomes
	\begin{equation} \label{evolution equation for d^N}
		\frac{d}{dt}d^N = \frac{1}{2i} \ \mathcal{B} d^N -  \lambda \  L^N\cdot d^N - ( \lambda + i)\mu \ G :: \left( (d^N\otimes\overline{d^N})^{\frac{p}{2}}\otimes d^N\right) ,
	\end{equation} 
	where $\mathcal{B}_{ij} = \beta_i\delta_{ij}$.  
	
	\subsubsection{Fixed point argument for the coefficients} \label{Fixed point argument for the coefficients}
	For a fixed $N$, a standard contraction mapping argument shows that \eqref{evolution equation for c^N} and \eqref{evolution equation for d^N} have unique solutions that are continuous in $[0,T]$. For a pair $(u^N_n,\psi^N_n)$, equivalently $(c^N_n,d^N_n)$, using Lemma \ref{existence of solutions to approx continuity equation}, we can find a solution $\rho^N_n$. Owing to the smoothness (in space) of the eigenfunctions used in the approximate velocity and wavefunction, we conclude that $u^N_n \xrightarrow[n\rightarrow\infty]{C^0_t C^1_x}u^N$ and $\psi^N_n \xrightarrow[n\rightarrow\infty]{C^0_t C^3_x}\psi^N$. Therefore, performing an iteration on the triplet $(c^N_n,d^N_n,\rho^N_n)$ and using Lemma \ref{convergence of solutions to approx continuity equation}, we conclude that the sequence $\rho^N_n$ converges to $\rho^N\in C^0_{[0,T]} C^0_{x}$. 
	
	
	
	\subsection{Compactness arguments} \label{compactness arguments}
	We now extract convergent subsequences from the a priori estimates in Section ~\ref{a priori estimates}. Beginning with the density, we know that $\rho^N \in C^0([0,T];C^0_x) \subset L^{\infty}(0,T;L^r_x)$ for $1\le r\le\infty$, meaning that 
	\begin{equation} \label{weak convergence of density}
		\rho^N \xrightharpoonup[L^{\infty}_t L^r_x]{\ast} \rho .
	\end{equation}
	Moreover, from \eqref{approximate continuity equation}, 
	\begin{equation} \label{partial_t rho^N uniformly bounded in L^2_t H^-1_x}
		\begin{aligned}
			\norm{\partial_t \rho^N}_{L^2_{[0,T]}H^{-1}_x} &\lesssim \norm{\nabla\cdot (u^N\rho^N)}_{L^2_{[0,T]}H^{-1}_x} + \norm{\Re (\overline{\psi^N}B^N\psi^N)}_{L^2_{[0,T]}H^{-1}_x} \\
			&\lesssim \norm{u^N\rho^N}_{L^2_{[0,T]}L^2_x} + \norm{ (\overline{\psi^N}B^N\psi^N)}_{L^2_{[0,T]}L^2_x} \\
			&\lesssim \norm{\sqrt{\rho^N}u^N}_{L^2_{[0,T]}L^2_x}\norm{\sqrt{\rho^N}}_{L^{\infty}_{[0,T]}L^{\infty}_x} + \norm{ \psi^N}_{L^{\infty}_{[0,T]}L^{\infty}_x}\norm{ B^N\psi^N}_{L^2_{[0,T]}L^2_x} .
		\end{aligned}    
	\end{equation} 
	The second inequality is due to the (compact) embedding $L^2_x \subset H^{-1}_x$ for $\T^2$. All the terms in the last line are finite (uniformly in $N$) by virtue of the a priori estimates. Therefore, using the Aubin-Lions-Simon lemma, we conclude the strong convergence 
	of a subsequence of the density as
	\begin{equation} \label{strong convergence of density}
		\rho^N \xrightarrow[]{C^0_t H^{-1}_x} \rho .
	\end{equation} 
	Consider a relabeled subsequence $\rho^N$ that strongly converges to $\rho$ in $C([0,T];H^{-1}_x)$, so that ~\eqref{truncated wavefunction definition} and ~\eqref{truncated velocity definition} are also appropriately relabeled. For a.e.~$s,t\in [0,T]$ and any $\omega \in H^1_x$,
	\begin{equation*}
		\begin{aligned}
			\langle \rho^N(t)-\rho^N(s),\omega \rangle_{H^{-1}\times H^1} &= \langle \int_s^t \partial_t \rho^N d\tau ,\omega \rangle_{H^{-1}\times H^1} \\ 
			&\le \int_s^t \norm{\partial_t \rho^N}_{H^{-1}_x} \norm{\omega}_{H^1_x}\le (t-s)^{\frac{1}{2}} \norm{\partial_t \rho^N}_{L^2_{[0,T]} H^{-1}_x} \norm{\omega}_{H^1_x} ,
		\end{aligned}
	\end{equation*} 
	showing that $\langle \rho^N(t),\omega \rangle_{H^{-1}\times H^1}$ is uniformly continuous in $[0,T]$, uniformly in $N$ due to ~\eqref{partial_t rho^N uniformly bounded in L^2_t H^-1_x}. Due to the embedding $H^1_x\subset L^r_x$ for all $1\le r<\infty$, we conclude, using the Arzela-Ascoli theorem, that $\rho^N$ is relatively compact in $C_w([0,T];L^r_x)$.
	
	We move on to the velocity. Based on the a priori estimates, we extract a subsequence of $u^N$ that weakly converges to $u\in L^{\infty}_{[0,T]}H^1_{\text{d},x} \cap L^2_{[0,T]}H^2_{\text{d},x}$, with $\partial_t u \in L^2_{[0,T]}L^2_{x}$. Applying the Lions-Magenes lemma (see ~\cite[Chapter 3]{Temam1977Navier-StokesAnalysis}), we deduce that $u\in C([0,T];H^1_{\text{d},x})$. Based on the $L^{\infty}_{t} L^{\infty}_x$ bound on the density, and the above strong convergences, it is easy to see that $\rho^N u^N$ and $\rho^N u^N \otimes u^N$ converge in $C([0,T];L^2_x)$ to $\rho u$ and $\rho u \otimes u$, respectively.
	
	Next, we consider the wavefunction. Again, we extract a subsequence that converges weakly to $\psi\in L^{\infty}_{[0,T]}H^{\frac{5}{2}}_{x} \cap L^2_{[0,T]}H^{\frac{7}{2}}_{x}$. From this and ~\eqref{NLS}, we have $\partial_t \psi \in L^2_{[0,T]}H^{\frac{3}{2}}_{x}$. Thus, the Lions-Magenes lemma yields $\psi\in C([0,T];H^{\frac{5}{2}}_{x})$. Additionally, we also have $B^N \psi^N \xrightarrow[]{C^0_t L^2_x} B\psi$, due to the regularity of $u$ and $\psi$.  
	
	As for the initial conditions, by construction itself (Section \ref{the initial density}), we have $\rho_0^N \xrightarrow[]{L^r_x} \rho_0$ for $1~\le ~r<~\infty$. Also, Lemma \ref{Q^N and P^N are convergent} states that $\psi_0^N$ and $u_0^N$ converge to $\psi_0$ and $u_0$ in $H^{\frac{5}{2}}_x$ and $H^{1}_{\text{d},x}$, respectively. For the momentum, we have
	\begin{equation}
		\norm{\rho_0^N u_0^N - \rho_0 u_0}_{L^2_x} \le \norm{\rho_0^N - \rho_0}_{L^r_x} \norm{u_0^N}_{L^{r'}_x} + \norm{\rho_0}_{L^{\infty}_x} \norm{u_0^N - u_0}_{L^2_x} ,
	\end{equation} 
	where $\frac{1}{r} + \frac{1}{r'} = \frac{1}{2}$. Using the embedding $H^{1}_x \subset L^{r'}_x$ to handle the velocity in the first term of the RHS, it is easy to see that the initial momentum converges in the $L^2_{x}$ norm.
	
	The approximate solutions $(\psi^N,u^N,\rho^N)$ are smooth enough to satisfy ~\eqref{weak solution wavefunction}--\eqref{weak solution density}. The aforementioned compactness results allow us to pass to the limit of $N\rightarrow\infty$ and arrive at the weak solutions $(\psi,u,\rho)$.

	\subsection{Renormalizing the density}
	At this point, we know that $\rho^N \xrightharpoonup{\ast} \rho$ in $L^{\infty}_t L^{\infty}_x$. We wish to use the technique of renormalization to extend this to $\rho^N \rightarrow \rho$ in $C^0_t L^r_x$, for $1\le r<\infty$. To achieve this, we will adapt a classical argument (see, for instance, Theorem 2.4 in \cite{Lions1996MathematicalMechanics}). We begin by defining a sequence of unit-mass mollifiers $\zeta_{h}(x) = \frac{1}{h^2} \zeta\left( \frac{x}{h} \right)$, where $h$ will eventually be taken to 0. Next, for a given weak solution $\rho \in L^{\infty}_t L^{\infty}_x$, we mollify ~\eqref{continuity} to obtain
	\begin{equation} \label{mollified continuity equation}
		\partial_t \rho_h + u\cdot\nabla\rho_h = \Psi_h + R_h ,
	\end{equation} 
	where $g_h := g * \zeta_h$, $\Psi := 2 \lambda\Re(\overline{\psi}B\psi)$, and $R_h := u\cdot\nabla\rho_h - \left( u\cdot\nabla\rho \right)_h$ is a commutator. We multiply this by $\eta'(\rho_h)$, for a $C^1$ function $\eta:\R\mapsto\R$. This yields
	\begin{equation} \label{mollified renormalized continuity equation}
		\partial_t \eta(\rho_h) + u\cdot\nabla\eta(\rho_h) = \eta'(\rho_h)\Psi_h + \eta'(\rho_h) R_h .
	\end{equation}
	The Sobolev embedding $H^2_x\subset W^{1,r_1}_x$ for any $r_1\in [1,\infty)$ implies that $u\in L^2_t W^{1,r_1}_x$. From Lemma ~2.3 in \cite{Lions1996MathematicalMechanics}, we note that $R_h$ vanishes in $L^2_t L^{r_1}_x$ (and also in $L^{\infty}_t L^2_x$) as $h\rightarrow 0$, by choosing $r_1>2$. Similarly, $\Psi_h$ converges to $\Psi$ in $C^0_t L^2_x$. Finally, note that $\eta'(\rho_h)$ is uniformly continuous since $\rho$ (and $\rho_h$) take values in a compact subset of $\R$. Therefore, using a test function $\sigma$, we may pass to the limit $h\rightarrow 0$ in \eqref{mollified renormalized continuity equation}. In other words, if $\rho$ is a weak solution of the continuity equation, then $\eta(\rho)$ solves (in a weak sense)
	\begin{equation} \label{renormalized continuity equation}
		\partial_t \eta(\rho) + u\cdot\nabla\eta(\rho) = \eta'(\rho)\Psi .
	\end{equation}
	This is the renormalized continuity equation.
	
	Taking the difference of \eqref{mollified continuity equation} for $h_1,h_2 >0$, we write the analog of ~\eqref{mollified renormalized continuity equation} for $\eta(\rho_{h_1} - \rho_{h_2})$, with $\eta(x) = x^{2n}$ where $n\in\N$. Integrating over $\T^2$ leads to
	\begin{equation*}
		\begin{aligned}
			\frac{d}{dt} \norm{\rho_{h_1} - \rho_{h_2}}_{L^{2n}_x}^{2n} &= \int_{\T^2} 2n \left(\rho_{h_1} - \rho_{h_2}\right)^{2n-1} \Big( (\Psi_{h_1} - \Psi_{h_2}) + (R_{h_1} - R_{h_2}) \Big) \\
			&\lesssim \norm{\rho_{h_1} - \rho_{h_2}}_{L^{2n}_x}^{2n-1} \left( \norm{\Psi_{h_1} - \Psi_{h_2}}_{L^{2n}_x} + \norm{R_{h_1} - R_{h_2}}_{L^{2n}_x} \right) ,
		\end{aligned}
	\end{equation*}
	which implies
	\begin{equation} \label{cauchy sequence of rho in L^r_x}
		\begin{aligned}
			\sup_{t\in [0,T]}\norm{\rho_{h_1} - \rho_{h_2}}_{L^{2n}_x} &\lesssim \norm{\rho(0)*\zeta_{h_1} - \rho(0)*\zeta_{h_2}}_{L^{2n}_x} \\
			&\quad + \int_0^T \left( \norm{\Psi_{h_1} - \Psi_{h_2}}_{L^{2n}_{x}} + \norm{R_{h_1} - R_{h_2}}_{L^{2n}_{x}} \right) .
		\end{aligned}
	\end{equation} 
	Since we know $\psi\in L^2_t H^2_x$ and $B\psi\in L^2_t H^{\frac{3}{2}}_x$, it follows from the Sobolev embedding and H\"older's inequalities that $\Psi = \overline{\psi} B\psi \in L^1_t L^{r_1}_x$ for any $r_1\in [1,\infty)$. Between this, the commutator estimate in Lemma 2.3 of ~\cite{Lions1996MathematicalMechanics}, and the boundedness of $\rho_0$, we find that all of the terms on the RHS of ~\eqref{cauchy sequence of rho in L^r_x} vanish as $h_1,h_2\rightarrow 0$, giving us a Cauchy sequence in $C([0,T];L^{2n}_x)$. Hence, $\rho_h$ converges to $\rho$ in $C([0,T];L^{2n}_x)$. We have, so far, proved that our ``original approximations" of the continuity equation $\rho^N$ converge in $C_w([0,T];L^r_x)$ to $\rho$, and that $\rho$ also belongs to $C([0,T];L^{2n}_x)$. To achieve what we set out to prove, i.e., that $\rho^N$ converges strongly in $C([0,T];L^r_x)$ to $\rho$, it remains to show that the $L^r_x$ norms are continuous in time. It is sufficient to illustrate this for $r=2$ (or $n=1$), in order to deduce it for the other values of $r$. Explicitly, if there is a sequence of times $t^N \rightarrow t$, then we need $\rho^N(t^N)$ to converge in $L^2_x$ to $\rho(t)$. Returning to \eqref{approximate continuity equation}, we look at its renormalized version with $\eta(x) = x^2$, and integrate over $\T^2$ (and then from $0$ to $t^N$) to get
	\begin{equation*}
		\int_{\T^2} (\rho^N(t^N))^2 = \int_{\T^2}(\rho_0^N)^2 + 2 \lambda \int_0^{t^N}\int_{\T^2} \rho^N \Re\left(\overline{\psi^N} B^N\psi^N\right) .
	\end{equation*}
	Since we know that $\rho\in C([0,T];L^2_x)$, we can do the same calculation with \eqref{continuity}, except over the time interval $0$ to $t$. This yields
	\begin{equation*}
		\int_{\T^2} (\rho(t))^2 = \int_{\T^2}(\rho_0)^2 + 2 \lambda \int_0^{t}\int_{\T^2} \rho \Re\left(\overline{\psi} B\psi \right).
	\end{equation*}
	Subtracting the last two equations, and taking the limit $N\rightarrow\infty$, we observe that the first terms on the RHS cancel (recall from Section \ref{the initial density} that $\rho_0^N \xrightarrow[]{L^2_x}\rho_0$). What remains is,
	\begin{equation*}
		\begin{aligned}
			\lim_{N\rightarrow\infty} \left( \int_{\T^2} (\rho^N(t^N))^2 - \int_{\T^2} (\rho(t))^2 \right) &= 2\lambda\Re \lim_{N\rightarrow\infty} \int_0^{t^N}\int_{\T^2} (\rho^N-\rho) \overline{\psi^N} B^N\psi^N \\ 
			&\quad + 2\lambda\Re \lim_{N\rightarrow\infty} \int_0^{t^N}\int_{\T^2} \rho \left(\overline{\psi^N}-\overline{\psi}\right) B^N\psi^N \\ 
			&\quad + 2\lambda\Re \lim_{N\rightarrow\infty} \int_0^{t^N}\int_{\T^2} \rho \overline{\psi} \left( B^N\psi^N - B\psi \right) \\ 
			&\quad + 2\lambda\Re \lim_{N\rightarrow\infty} \int_t^{t^N}\int_{\T^2} \rho \overline{\psi} B\psi.
		\end{aligned}
	\end{equation*}
	Thanks to the uniform boundedness of $\overline{\psi^N} B^N\psi^N$ in $L^1_{[0,T]} H^{\frac{3}{2}}_x$, we can use the strong convergence in ~\eqref{strong convergence of density} to handle the first term on the RHS. The second and third terms follow from simple H\"older's inequalities, and the strong convergence of $\psi^N$ of $B^N\psi^N$. Finally, the last term is integrable on $[0,T]$, so as $t^N\rightarrow t$, it vanishes. In summary,
	\begin{equation} \label{strong convergence of density in C^0_t L^2_x}
		\rho^N \xrightarrow[]{C^0_t L^2_x} \rho ,
	\end{equation}
	which, along with the weak-in-time continuity deduced earlier, implies strong convergence of $\rho^N$ to $\rho$ in $C^0_t L^{2n}_x$ for all $n\in\N$. Interpolating between Lebesgue norms extends this result to $C^0_t L^r_x$ for all $r\in [1,\infty)$.
	

	\subsection{The energy equality}
	The smooth approximations to the weak solutions satisfy an energy equation, given by ~\eqref{energy equality for weak solutions}, i.e., 
	\begin{equation} \label{energy equality for smooth approximations}
		\begin{aligned}
			&\left( \frac{1}{2}\norm{\sqrt{\rho^N (t)}u^N (t)}_{L^2_x}^2 + \frac{1}{2}\norm{\nabla\psi^N(t)}_{L^2_x}^2 + \frac{2\mu}{p+2}\norm{\psi^N(t)}_{L^{p+2}_x}^{p+2} \right) \\ 
			&\quad + \nu\norm{\nabla u^N}_{L^2_{[0,t]}L^2_x}^2 + 2 \lambda\norm{B^N\psi^N}_{L^2_{[0,t]}L^2_x}^2 \\ 
			&= \frac{1}{2}\norm{\sqrt{\rho_0^N}u_0^N}_{L^2_x}^2 + \frac{1}{2}\norm{\nabla\psi_0^N}_{L^2_x}^2 + \frac{2\mu}{p+2}\norm{\psi_0^N}_{L^{p+2}_x}^{p+2} ,
		\end{aligned}
	\end{equation} 
	for a.e.~$t\in [0,T]$. From our choice of the initial conditions and their approximations (see Section \ref{the initial conditions}), we can ensure that as $N\rightarrow\infty$, the RHS converges to the initial energy $E_0$ defined in \eqref{E0 definition}. Indeed, for the first term,
	\begin{equation} \label{convergence of initial kinetic energy}
		\begin{aligned} 
			\abs{\norm{\sqrt{\rho_0^N}u_0^N}_{L^2_x}^2 - \norm{\sqrt{\rho_0}u_0}_{L^2_x}^2} &= \abs{\int_{\T^2} \rho_0^N \abs{u_0^N}^2 - \rho_0 \abs{u_0}^2} \\
			&\lesssim \norm{\rho_0^N - \rho_0}_{L^2_x} \norm{u_0^N}_{L^4_x}^2 + \norm{\rho_0}_{L^{\infty}_x} \norm{u_0^N + u_0}_{L^2_x} \norm{u_0^N - u_0}_{L^2_x} \\
			&\xrightarrow[]{N\rightarrow\infty} 0 .
		\end{aligned}
	\end{equation}
	Moreover, based on the results of Section \ref{compactness arguments}, we can conclude that all the terms on the LHS of ~\eqref{energy equality for smooth approximations} converge strongly to the corresponding terms with the approximate solutions replaced by the weak solution. The first term on the LHS can be dealt with the same way as the first term on the RHS in \eqref{convergence of initial kinetic energy}, by simply including a $\sup_t$ outside the absolute values.
	\qed
	
	This completes the construction of the solutions. Together with the global/almost global estimates from Section ~\ref{a priori estimates}, we can conclude the results of Theorems ~\ref{global existence} and ~\ref{almost global existence}.

	\addtocontents{toc}{\protect\setcounter{tocdepth}{0}}
	
	\section*{Acknowledgments}
    J.J. and I.K. were supported by the NSF grants DMS-2009458 and DMS-2205493, respectively. The authors appreciate the comments of the referees and the editors which helped improve the manuscript.
	
	
	\addtocontents{toc}{\protect\setcounter{tocdepth}{2}}
	
	\bibliographystyle{alpha}
	\bibliography{references.bib}
	
\end{document}